\documentclass[12pt,reqno]{amsart}

\usepackage{mathdots}

\usepackage{color}

\usepackage{pdfsync}

\usepackage{enumitem}

\usepackage{wasysym}

\usepackage{amssymb}

\usepackage{mathrsfs}

\DeclareMathAlphabet{\mathpzc}{OT1}{pzc}{m}{it}

\usepackage[all]{xy}

\usepackage{tikz}
\usetikzlibrary{arrows,matrix,decorations.pathmorphing,decorations.pathreplacing,positioning,shapes.geometric,shapes.misc,decorations.markings,decorations.fractals,calc,patterns}

\usepackage{graphicx}
\usepackage{float}
\usepackage[bottom]{footmisc}
\usepackage{moreenum}
\entrymodifiers={+!!<0pt,\fontdimen22\textfont2>}

\def\cD{\mathscr{D}}

\def\cF{\mathscr{F}}
\def\cG{\mathscr{G}}

\def\cL{\mathscr{L}}
\def\cM{\mathscr{M}}

\def\cW{\mathscr{W}}

\def\add{\operatorname{add}}

\def\adots{\mathinner{\mkern1mu\raise1.0pt\vbox{\kern7.0pt\hbox{.}}\mkern2mu\raise4.0pt\hbox{.}\mkern2mu\raise7.0pt\hbox{.}\mkern1mu}}

\def\ast{{\textstyle *}}
\def\astsmall{{\scriptstyle *}}

\def\dddots{\mathinner{\mkern1mu\raise10.0pt\vbox{\kern7.0pt\hbox{.}}\mkern2mu\raise5.3pt\hbox{.}\mkern2mu\raise1.0pt\hbox{.}\mkern1mu}}
\def\dddotssmall{\mathinner{\mkern1mu\raise7.0pt\vbox{\kern7.0pt\hbox{.}}\mkern-1mu\raise4pt\hbox{.}\mkern-1mu\raise1.0pt\hbox{.}\mkern1mu}}

\def\Db{\cD^{\operatorname{b}}}

\def\dim{\operatorname{dim}}

\def\dual{\operatorname{D}}
\def\End{\operatorname{End}}

\def\Ext{\operatorname{Ext}}

\def\gldim{\operatorname{gldim}}

\def\H{\operatorname{H}}

\def\Hom{\operatorname{Hom}}

\def\Image{\operatorname{Im}}

\def\Ker{\operatorname{Ker}}

\newcommand\LTensor[1]{\overset{{\rm L}}{\underset{#1}{\otimes}}}
\def\mod{\operatorname{mod}}
\def\Mod{\operatorname{Mod}}

\def\opp{\operatorname{op}}

\def\pd{\operatorname{projdim}}

\def\rad{\operatorname{rad}}

\def\RHom{\operatorname{RHom}}

\def\SL2{\operatorname{SL}_2}

\def\Tor{\operatorname{Tor}}

\numberwithin{equation}{section}

\renewcommand{\labelenumi}{(\roman{enumi})}

\newtheorem{Lemma}{Lemma}[section]
\newtheorem{Proposition}[Lemma]{Proposition}

\theoremstyle{definition}
\newtheorem{Definition}[Lemma]{Definition}
\newtheorem{Setup}[Lemma]{Setup}

\newtheorem{Construction}[Lemma]{Construction}
\newtheorem{Remark}[Lemma]{Remark}

\newtheorem{Example}[Lemma]{Example}

\newtheorem*{bfhpg*}{}

\newenvironment{VarDescription}[1]%
  {\begin{list}{}{%
    \settowidth{\labelwidth}{\textbf{#1:}}%
    \setlength{\leftmargin}{\labelwidth}\addtolength{\leftmargin}{\labelsep}}}%
  {\end{list}}

\newcommand{\sHom}{\underline{\Hom}}
\def\Cok{\operatorname{Cok}}

\begin{document}

\setlength{\parindent}{0pt}
\setlength{\parskip}{7pt}
%The default \baselineskip is close to 4.8mm
%\setlength{\baselineskip}{5.8mm}

\title[Wide subcategories of $d$-cluster tilting subcategories]{Wide subcategories of $d$-cluster tilting subcategories}

\author{Martin Herschend}
\address{Department of Mathematics, Uppsala University, Box 480, 751
  06 Uppsala, Sweden}
\email{martin.herschend@math.uu.se}
%\urladdr{}

\author{Peter J\o rgensen}
\address{School of Mathematics and Statistics,
Newcastle University, Newcastle upon Tyne NE1 7RU, United Kingdom}
\email{peter.jorgensen@ncl.ac.uk}
\urladdr{http://www.staff.ncl.ac.uk/peter.jorgensen}

\author{Laertis Vaso}
\address{Department of Mathematics, Uppsala University, Box 480, 751
  06 Uppsala, Sweden}
\email{laertis.vaso@math.uu.se}
%\urladdr{}

%\thanks{Date: \today. A thank you would go here}



%05E10: Combinatorial aspects of representation theory
%05E15: Combinatorial\ aspects of groups and algebras
%05E40: Combinatorial aspects of commutative algebra
%05E45: Combinatorial aspects of simplicial complexes
%13D25: Complexes
%13F60: Cluster algebras
%16E10: Homological dimension
%16E45: Differential graded algebras and applications
%16G10: Representations of Artinian rings 
%16G60: Representation type (finite, tame, wild, etc.) 
%16G70: Auslander-Reiten sequences (almost split sequences) and
%       Auslander-Reiten quivers
%16S90: Torsion theories; radicals on module categories
%18A20: Epimorphisms, monomorphisms, special classes of morphisms, null morphisms
%18E10: Exact categories, abelian categories
%18E30: Derived categories, triangulated categories
%18E35: Localization of categories
%18E40: Torsion theories, radicals
%18G05: Projectives and injectives
%18G35: Chain complexes
%18G99: Homological algebra: None of the above, but in this section 
%55P62: Rational homotopy theory

\begin{abstract} 

  A subcategory of an abelian category is wide if it is closed under
  sums, summands, kernels, cokernels, and extensions.  Wide
  subcategories provide a significant interface between representation
  theory and combinatorics.

\bigskip
\noindent
  If $\Phi$ is a finite dimensional algebra, then each functorially finite wide subcategory of $\mod( \Phi )$ is of the form $\phi_{ \ast }\big( \mod( \Gamma ) \big)$ in an essentially unique way, where $\Gamma$ is a finite dimensional algebra and $\Phi \stackrel{ \phi }{ \longrightarrow } \Gamma$ is an algebra epimorphism satisfying $\Tor^{ \Phi }_1( \Gamma,\Gamma ) = 0$.

\bigskip
\noindent
  Let $\cF \subseteq \mod( \Phi )$ be a $d$-cluster tilting subcategory
  as defined by Iyama.  Then $\cF$ is a $d$-abelian category as defined by Jasso, and we call a subcategory of $\cF$ wide if it is closed under
  sums, summands, $d$-kernels, $d$-cokernels, and $d$-extensions.  We
   generalise the above description of wide subcategories to this setting: Each
  functorially finite wide subcategory of $\cF$ is of the form
  $\phi_{ \ast }( \cG )$ in an essentially unique way, where
  $\Phi \stackrel{ \phi }{ \longrightarrow } \Gamma$ is an algebra
  epimorphism satisfying $\Tor^{ \Phi }_d( \Gamma,\Gamma ) = 0$, and
  $\cG \subseteq \mod( \Gamma )$ is a $d$-cluster tilting subcategory.

\bigskip
\noindent
  We illustrate the theory by computing the wide subcategories of
  some $d$-cluster tilting subcategories $\cF \subseteq \mod( \Phi )$
  over algebras of the form $\Phi = kA_m / (\rad\,kA_m )^{ \ell }$.

\end{abstract}

\maketitle

\begin{center}
{\em Dedicated to Idun Reiten on the occasion of her 75th birthday}
\end{center}
\bigskip

\section{Introduction}
\label{sec:introduction}

Let $d \geqslant 1$ be an integer.  This paper introduces and studies wide subcategories of $d$-abelian categories as defined by Jasso.  The main examples of $d$-abelian categories are $d$-cluster tilting subcategories as defined by Iyama.  For $d=1$ these categories are abelian in the classic sense, hence our theory generalises the theory of wide subcategories of abelian categories.

\subsection{Outline}
Let $\cM$ be an abelian category.  A full subcategory $\cW \subseteq \cM$ is called {\em wide} if it is closed under sums, summands, kernels, cokernels, and extensions.  Wide subcategories have been studied by a number of authors because of their combinatorial and geometrical significance, see \cite{Br}, \cite{GP}, \cite{IT}, \cite{Kr}, \cite{MS}, \cite{Sc}.  

Now let $\cM$ be a $d$-abelian category in the sense of Jasso, see Definition \ref{def:d-abelian}.  While $\cM$ does not have kernels and cokernels, it does have $d$-kernels and $d$-cokernels; these are complexes of $d$ objects with a weaker universal property than kernels and cokernels.   The main examples of $d$-abelian categories are $d$-cluster tilting subcategories in the sense of Iyama, see Definition \ref{def:d-cluster_tilting}. There is an extensive theory of $d$-cluster tilting subcategories, including a wealth of examples linked to combinatorics and geometry, see \cite{DI}, \cite{HI1}, \cite{HI2}, \cite{I1}, \cite{I3}, \cite{IO}.

We say that a full subcategory $\cW \subseteq \cM$ is {\em wide} if it is closed under sums, summands, $d$-kernels, $d$-cokernels, and $d$-extensions, see Definition \ref{def:wide}.  This paper studies wide subcategories of $d$-abelian categories.  Theorems A and B show that wide subcategories of $d$-cluster tilting subcategories are intimately related to algebra homomorphisms and certain homologically well behaved objects.  Theorem C verifies the combinatorial significance of wide subcategories by giving a combinatorial description of the wide subcategories of some fundamental $d$-cluster tilting subcategories.

\subsection{Classic background: Wide subcategories of module categories}
Algebra homomorphisms are an important theoretical tool for studying wide subcategories of module categories.  Let $k$ be an algebraically closed field.  If $\Phi \stackrel{ \phi }{ \longrightarrow } \Gamma$ is a homomorphism of finite dimensional $k$-algebras, then there is a functor
$\mod( \Gamma ) \stackrel{ \phi_{ \astsmall } }{ \longrightarrow }
\mod( \Phi )$ between the categories of finitely generated right
modules, given by restriction of scalars from $\Gamma$ to $\Phi$.
The essential image $\phi_{ \ast }\big( \mod( \Gamma ) \big)$ is a full subcategory of $\mod( \Phi )$, and it turns out that each functorially finite wide subcategory has this form.

{\bf Classic Theorem. }
{\em
Let $\Phi$ be a finite dimensional $k$-algebra.  There is a bijection 
\[
  \left\{\!\!
    \begin{array}{ll}
      \mbox{ equivalence classes of pseudoflat epimorphisms } \\
      \mbox{ of finite dimensional $k$-algebras $\Phi \stackrel{ \phi }{ \longrightarrow } \Gamma$ }\\
    \end{array}
  \!\right\}
  \rightarrow
  \left\{\!\!
    \begin{array}{ll}
      \mbox{ functorially finite wide } \\
      \mbox{ subcategories of $\mod( \Phi )$ }
    \end{array}
  \!\!\right\},
\]
sending the equivalence class of
$\Phi \stackrel{ \phi }{ \longrightarrow } \Gamma$ to
$\phi_{ \ast }\big( \mod( \Gamma ) \big)$.
\hfill $\Box$
}

Here, an epimorphism of finite dimensional $k$-algebras
$\Phi \stackrel{ \phi }{ \longrightarrow } \Gamma$ is an algebra
homomorphism which is an epimorphism in the category of rings.  It is
pseudoflat if $\Tor^{ \Phi }_1( \Gamma,\Gamma ) = 0$; this notion is due to \cite{BD}.  Equivalence of epimorphisms is defined straightforwardly, and we refer to Definition \ref{def:env} for the notion of functorial finiteness.

The Classic Theorem is due to the efforts of several authors.  It can be obtained by combining Iyama's result \cite[thm.\ 1.6.1(2)]{I4} with Schofield's \cite[Theorem 4.8]{Sc}.  It is stated explicitly by Marks and \v{S}\v{t}ov\'{\i}\v{c}ek in \cite[prop.\ 4.1]{MS}.  There are important earlier results by 
Gabriel and de la Pe{\~n}a, who worked in the category $\Mod( \Phi )$ of all right modules and considered subcategories closed under products, coproducts, kernels, and cokernels, see \cite{GP}.  There are related results by Geigle and Lenzing, see \cite{GL}.

\subsection{This paper: Wide subcategories of $d$-cluster tilting subcategories}
Let $\cF \subseteq \mod( \Phi )$ be a $d$-cluster tilting subcategory as defined by Iyama, see Definition \ref{def:d-cluster_tilting}, where $\Phi$ is a finite dimensional $k$-algebra.  Then \cite[thm.\ 3.16]{J} implies that $\cF$ is $d$-abelian in the sense of Jasso, see Definition \ref{def:d-abelian}.  
We say that $( \Phi,\cF )$ is a {\em $d$-homological pair}, and view $\cF$ as a ``higher'' analogue of $\mod( \Phi )$.  Indeed, if $d=1$, then $\cF = \mod( \Phi )$ is the unique choice.

Our first main result is the following generalisation of the above Classic Theorem, which appears as the case $d = 1$.

{\bf Theorem A. }
{\em
Let $( \Phi,\cF )$ be a $d$-homological pair.  There is a bijection
\[
  \left\{\!\!
    \begin{array}{ll}
      \mbox{ equivalence classes of $d$-pseudoflat epimorphisms } \\
      \mbox{ of $d$-homological pairs $( \Phi,\cF ) \stackrel{ \phi }{ \longrightarrow } ( \Gamma,\cG )$ }
    \end{array}
  \!\right\}
  \rightarrow
  \left\{\!\!
    \begin{array}{ll}
      \mbox{ functorially finite wide } \\
      \mbox{ subcategories of $\cF$ }
    \end{array}
  \!\!\right\},
\]
sending the equivalence class of
$( \Phi,\cF ) \stackrel{ \phi }{ \longrightarrow } ( \Gamma,\cG )$ to
$\phi_{ \ast }( \cG )$.
\hfill $\Box$
}

Here, an epimorphism of $d$-homological pairs
$( \Phi,\cF ) \stackrel{ \phi }{ \longrightarrow } ( \Gamma,\cG )$ is
an epimorphism of algebras
$\Phi \stackrel{ \phi }{ \longrightarrow } \Gamma$ such that
$\phi_{ \ast }( \cG ) \subseteq \cF$.  It is $d$-pseudoflat if
$\Tor^{ \Phi }_d ( \Gamma,\Gamma ) = 0$.  Equivalence of epimorphisms
is again defined straightforwardly.  See Definition
\ref{def:morphisms_of_d-homological_pairs}.

The notion of a morphism $d$-homological pairs appears to be natural.  For instance, a morphism gives three functors which form two adjoint pairs between $\cF$ and $\cG$, see Proposition \ref{pro:adjoints}.  This may be useful outside the theory of wide subcategories.

The methods which lead to Theorem A also give the following result which applies to a more special setup, but is well suited to computations.

{\bf Theorem B. }
{\em
Let $( \Phi,\cF )$ be a $d$-homological pair and $\cW \subseteq \cF$ a full
subcategory closed under isomorphisms in $\cF$.  Let $s \in \cW$ be a module.  

Set $\Gamma = \End_{\Phi}(s)$ so $s$ acquires the structure ${}_{ \Gamma }s_{ \Phi }$.  Assume the following: 
\begin{enumerate}
\setlength\itemsep{4pt}

  \item  The projective dimension satisfies $\pd( s_{ \Phi } ) < \infty$.

  \item  $\Ext_{\Phi}^{\geqslant 1}(s,s) = 0$.

  \item  Each $w \in \cW$ permits an exact sequence 
$
  0
  \rightarrow p_m
  \rightarrow \cdots
  \rightarrow p_1
  \rightarrow p_0
  \rightarrow w
  \rightarrow 0
$
in $\mod( \Phi )$ with $p_i \in \add( s )$.

  \item $\cG := \Hom_\Phi( s,\cW ) \subseteq \mod( \Gamma )$ is
    $d$-cluster tilting. 

\end{enumerate}
Then $\cW$ is a wide subcategory of $\cF$ and there is an equivalence
of categories
\[
\tag*{$\Box$}
  - \underset{ \Gamma }{ \otimes } s : \cG \rightarrow \cW.
\]
}

The utility of Theorem B will be illustrated by showing a combinatorial description of the wide subcategories of some fundamental $d$-cluster tilting subcategories.  Let $Q$ be the quiver
\[
  m \rightarrow \cdots \rightarrow 2 \rightarrow 1
\]  
with $m \geqslant 3$.  Assume that $d$ and $\ell$ are positive integers
such that $d$ is even and
\[
  \frac{ m-1 }{ \ell } = \frac{ d }{ 2 }.
\]
It is shown in \cite{V} that $\Phi = kQ / ( \rad kQ )^{ \ell }$ has
global dimension $d$ and that $\mod( \Phi )$ has a unique $d$-cluster
tilting subcategory $\cF$ consisting of the sums of projective and
injective mo\-du\-les.  We will establish the following.

{\bf Theorem C. }
{\em
Let $\cW \subseteq \cF$ be an additive subcategory which is not
semisimple. Then $\cW$ is wide if and only if $\cW$ is
${\ell}$-periodic (see Definition \ref{def:l-periodic}).  In
particular, there are exactly $2^{\ell}-{\ell}-1$ non-semisimple wide
subcategories of $\cF$.
\hfill $\Box$
}

Note that we also characterise the semisimple wide subcategories of $\cF$
in Proposition \ref{semisimple}.

The paper is organised as follows: Section \ref{sec:background}
collects a number of definitions which are used throughout.  Section
\ref{sec:Proof_of_Thm_B} proves Theorem B.  Sections \ref{sec:wide}
and \ref{sec:d-homological_pairs} establish some properties of wide
subcategories of $d$-abelian categories and morphisms of
$d$-homological pairs.  Sections \ref{sec:Proof_of_Thm_A} and
\ref{sec:Proof_of_Thm_C} prove Theorems A and C.

\section{Definitions}
\label{sec:background}

This section is a reminder of some fundamental definitions.  They are all well known except Definition \ref{def:wide}.

\begin{Setup}
\label{set:blanket1}
Throughout, $k$ is an algebraically closed field, $\Phi$, $\Gamma$,
$\Lambda$ are finite dimensional $k$-algebras, and $d \geqslant 1$ is
a fixed integer.  Some associated categories are:
\begin{itemize}
\setlength\itemsep{4pt}

  \item  The category of finitely generated right $\Phi$-modules:
    $\mod( \Phi )$.

  \item  The derived category of right $\Phi$-modules: $\cD( \Phi )$.

  \item  The bounded derived category of finitely
generated right $\Phi$-modules: $\Db( \mod\,\Phi )$.

\end{itemize}
Module structures are sometimes shown with subscripts:
${}_{ \Gamma }s_{ \Phi }$ shows that $s$ is a $\Gamma$-$\Phi$-bimodule.  Duality over $k$ is denoted $\dual( - ) = \Hom_k( -,k )$. 
\hfill $\Box$
\end{Setup}

\begin{Definition}
[Monomorphisms and epimorphisms]
Let $\cF$ be a category.
\begin{itemize}
\setlength\itemsep{6pt}

  \item A {\em monomorphism in $\cF$} is a morphism
  $f^0 \rightarrow f^1$ such that each object $f \in \cF$ induces an
  injection $\cF( f,f^0 ) \rightarrow \cF( f,f^1 )$.

  \item An {\em epimorphism in $\cF$} is a morphism
  $f_1 \rightarrow f_0$ such that each object $f \in \cF$ induces an
  injection $\cF( f_0,f ) \rightarrow \cF( f_1,f )$.

\end{itemize}
If $\cW \subseteq \mod( \Phi )$ is a full subcategory, then each injection of modules in $\cW$ is clearly a monomorphism in $\cW$, but there may be more monomorphisms in $\cW$ than injections.  The analogous statement holds for surjections versus epimorphisms. 
\hfill $\Box$
\end{Definition}

\begin{Definition}
[$d$-abelian categories, see {\cite[def.\ 3.1]{J}}]
\label{def:d-abelian}
Let $\cM$ be an additive category.

A {\em $d$-kernel} of a morphism $m_1 \rightarrow m_0$ is a diagram
$m_{ d+1 } \rightarrow \cdots \rightarrow m_1$ such
that
\[
  0
  \rightarrow \cM( m,m_{ d+1 } )
  \rightarrow \cdots
  \rightarrow \cM( m,m_2 )
  \rightarrow \cM( m,m_1 )
  \rightarrow \cM( m,m_0 )
\]
is exact for each $m \in \cM$.  The notion of {\em $d$-cokernel} is
dual.

A {\em $d$-exact sequence} is a diagram $m_{ d+1 } \rightarrow \cdots
\rightarrow m_0$ such that 
\begin{align*}
  m_{ d+1 } \rightarrow \cdots \rightarrow m_1 
    & \mbox{ is a $d$-kernel of $m_1 \rightarrow m_0$ and} \\
  m_d \rightarrow \cdots \rightarrow m_0
    & \mbox{ is a $d$-cokernel of $m_{ d+1 } \rightarrow m_d$. }
\end{align*}

The category $\cM$ is called {\em $d$-abelian} if it satisfies:
\begin{VarDescription}{(A2${}^{\opp}$)\quad}
\setlength\itemsep{4pt}

  \item[(A0)\:]  $\cM$ has split idempotents.

  \item[(A1)\:]  Each morphism in $\cM$ has a $d$-kernel and a $d$-cokernel.

  \item[(A2)\:]  If $m_{ d+1 } \rightarrow m_d$ is a monomorphism which
    has a $d$-cokernel $m_d \rightarrow \cdots \rightarrow m_0$, then
    $m_{ d+1 } \rightarrow \cdots \rightarrow m_0$ is a $d$-exact
    sequence. 

  \item[(A2${}^{\opp}$)\:]  The dual of (A2).

\end{VarDescription}
Conditions (A2) and (A2${}^{\opp}$) can be replaced with:
\begin{VarDescription}{(A2'${}^{\opp}$)\quad}
\setlength\itemsep{4pt}

  \item[(A2')\:] If $m_{ d+1 } \rightarrow m_d$ is a monomorphism, then
  there exists a $d$-exact sequence
  $m_{ d+1 } \rightarrow \cdots \rightarrow m_0$.

  \item[(A2'${}^{\opp}$)\:]  The dual of (A2').

\end{VarDescription}
Note that $1$-kernels, $1$-cokernels, $1$-exact sequences, and
$1$-abelian categories are synonymous with kernels, cokernels, short
exact sequences, and abelian categories. 
\hfill $\Box$
\end{Definition}

\begin{Definition}
[$d$-cluster tilting subcategories, see {\cite[def.\ 1.1]{I1}}]
\label{def:d-cluster_tilting}
A full subcategory $\cF \subseteq \mod( \Phi )$ is called {\em
  $d$-cluster tilting} if it is functorially finite (see Definition
\ref{def:env}) and satisfies
\begin{align*}
  \cF
  & = \{ f \in \mod( \Phi ) \mid
         \Ext_{ \Phi }^1( \cF,f ) = \cdots = \Ext_{ \Phi }^{ d-1 }( \cF,f ) = 0 \} \\
%\tag*{$\Box$}
  & = \{ f \in \mod( \Phi ) \mid
         \Ext_{ \Phi }^1( f,\cF ) = \cdots = \Ext_{ \Phi }^{ d-1 }( f,\cF ) = 0 \}.
\end{align*}
Note that if $d = 1$ then $\cF = \mod( \Phi )$.
\hfill $\Box$
\end{Definition}

\begin{Definition}
[$d$-homological pairs]
\label{def:d-homological_pair}
We say that $( \Phi,\cF )$ is a {\em $d$-homological pair} if $\cF \subseteq \mod( \Phi )$ is a $d$-cluster tilting subcategory.
\hfill $\Box$
\end{Definition}

\begin{Remark}
\label{rmk:d-homological_pairs}
If $( \Phi,\cF )$ is a $d$-homological pair, then $\cF$ is a $d$-abelian category by \cite[thm.\ 3.16]{J}.  Moreover, it is clear that $\cF$ is an essentially small $k$-linear $\Hom$-finite category,  and $p = \Phi_{ \Phi }$ is a {\em projective generator}.  This means that $p$ is a categorically projective object (see Definition \ref{def:projectives}) such that each $m \in \cM$ permits an epimorphism $p_0 \rightarrow m$ with $p_0 \in \add( p )$.

Conversely, if $\cM$ is an essentially small $k$-linear $\Hom$-finite $d$-abelian category with a projective generator $p$, then by \cite[thm.\ 3.20]{J} there is a $d$-homological pair $( \Phi,\cF )$ with $\Phi = \End_{ \cM }( p )$ such that $\cM$ is equivalent to $\cF$.
\hfill $\Box$
\end{Remark}

\begin{Definition}
[Additive subcategories]
In an additive category, an {\em additive subcategory} is a full subcategory closed under direct sums, direct summands, and isomorphisms in the ambient category.
\end{Definition}

\begin{Definition}
[Wide subcategories]
\label{def:wide}
An additive subcategory $\cW$ of a $d$-abelian category $\cM$ is
called {\em wide} if it satisfies the following conditions:
\begin{enumerate}
\setlength\itemsep{4pt}

  \item  Each morphism in $\cW$ has a $d$-kernel in $\cM$ which
    consists of objects from $\cW$.

  \item  Each morphism in $\cW$ has a $d$-cokernel in $\cM$ which
    consists of objects from $\cW$.

  \item  Each $d$-exact sequence in $\cM$,
\[
  0
  \rightarrow w'
  \rightarrow m_d
  \rightarrow \cdots
  \rightarrow m_1
  \rightarrow w''
  \rightarrow 0,
\]
with $w',w'' \in \cW$ is Yoneda equivalent to a $d$-exact sequence in
$\cM$, 
\[
  0
  \rightarrow w'
  \rightarrow w_d
  \rightarrow \cdots
  \rightarrow w_1
  \rightarrow w''
  \rightarrow 0,
\]
with $w_i \in \cW$ for each $i$.
\hfill $\Box$

\end{enumerate}
\end{Definition}

\section{Proof of Theorem B}
\label{sec:Proof_of_Thm_B}

This section proves Theorem B.

\begin{Definition}
[Resolutions]
\label{def:resolution}
Let $\cF \subseteq \mod( \Phi )$ be an additive subcategory and
$m \in \mod( \Phi )$ an object.  An {\em augmented left
  $\cF$-resolution} of $m$ is a sequence
\[
  \cdots \rightarrow f_2 \rightarrow f_1 \rightarrow
  f_0 \rightarrow m \rightarrow 0
\]
with $f_i \in \cF$ for each $i$, which becomes exact under
$\Hom_{ \Phi }( f,- )$ for each $f \in \cF$.  Then
\[
  \cdots 
  \rightarrow f_2
  \rightarrow f_1
  \rightarrow f_0
  \rightarrow 0
  \rightarrow \cdots
\]
is called a {\em left $\cF$-resolution} of $m$.
\hfill $\Box$
\end{Definition}

The following lemma is due to Iyama.

\begin{Lemma}
\label{lem:resolution}
Let $( \Phi,\cF )$ be a $d$-homological pair.  Each
$m \in \mod( \Phi )$ has an augmented left $\cF$-resolution of the form
\[
  \cdots \rightarrow 0 \rightarrow f_{ d-1 } \rightarrow f_{ d-2 }
  \rightarrow \cdots \rightarrow f_1 \rightarrow f_0 \rightarrow m \rightarrow 0.
\]
Conversely, an exact sequence of this form with $f_i \in \cF$ for each
$i$ is an augmented left $\cF$-resolution of $m$.
\end{Lemma}

\begin{Lemma}
\label{lem:homdim}
Let $( \Phi,\cF )$ be a $d$-homological pair, $x$ a finitely generated left $\Phi$-module, $\ell \geqslant 0$ an integer.  Then
\begin{equation}
\label{equ:Tor_assumption}
  \Tor^{ \Phi }_{ >\ell }( \cF,x ) = 0
\end{equation}
if and only if $\pd( {}_{ \Phi }x ) \leqslant \ell$.
\end{Lemma}

\begin{proof}
It is clear that $\pd( {}_{ \Phi }x ) \leqslant \ell$ implies
\eqref{equ:Tor_assumption}.  

Conversely, assume \eqref{equ:Tor_assumption}.  By the dual of Lemma
\ref{lem:resolution} each $z \in \mod( \Phi )$ permits an exact sequence
$0 \rightarrow z \rightarrow f^0 \rightarrow \cdots \rightarrow f^{
  d-1 } \rightarrow 0$ with $f^i \in \cF$.  The complex
\[
  f
  = \cdots
  \rightarrow 0
  \rightarrow f^0
  \rightarrow \cdots
  \rightarrow f^{ d-1 }
  \rightarrow 0
  \rightarrow \cdots
\]
is isomorphic to $z$ in $\Db( \mod\,\Phi )$ and we must prove $\Tor^{ \Phi }_{ >\ell }( f,x ) = 0$.   

To show this, let $j$ be an integer with $0 \leqslant j \leqslant d-1$
and consider a complex
\[
  e
  = \cdots
  \rightarrow 0
  \rightarrow e^j
  \rightarrow \cdots
  \rightarrow e^{ d-1 }
  \rightarrow 0
  \rightarrow \cdots
\]
with $e^i \in \cF$ for each $i$.  We prove by descending induction on
$j$ that $\Tor^{ \Phi }_{ >\ell }( e,x ) = 0$.  If $j = d-1$ then $e$
is concentrated in cohomological degree $j = d-1$.  This means that
$e = \Sigma^{ -j }e^j$ so $\Tor^{ \Phi }_{ >\ell }( e,x )$ is
isomorphic to
\begin{equation}
\label{equ:Tor_vanishing}  
  \Tor^{ \Phi }_{ >\ell }( \Sigma^{ -j }e^j,x )
  = \Tor^{ \Phi }_{ >\ell+j }( e^j,x ) \\
  = 0
\end{equation}
where the last equality is by \eqref{equ:Tor_assumption}.  If
$j < d-1$ then hard truncation gives a triangle
$e' \rightarrow e \rightarrow e''$ in $\Db( \mod\,\Phi )$ with
\[
  \begingroup%\begingroup \endgroup mean that the changed arraycolsep
             %doesn't migrate to the rest of the document
  \setlength{\arraycolsep}{2pt}
  \begin{array}{ccccccccccccccccc}
    e' & = & \cdots
    & \rightarrow & 0
    & \rightarrow & 0
    & \rightarrow & e^{ j+1 }
    & \rightarrow & \cdots
    & \rightarrow & e^{ d-1 }
    & \rightarrow & 0
    & \rightarrow & \cdots, \\[2mm]
    e'' & = & \cdots
    & \rightarrow & 0
    & \rightarrow & e^j
    & \rightarrow & 0
    & \rightarrow & \cdots
    & \rightarrow & 0
    & \rightarrow & 0
    & \rightarrow & \cdots.
  \end{array}
  \endgroup
\]
This gives a long exact sequence consisting of pieces
\[
  \Tor^{ \Phi }_i( e',x )
  \rightarrow \Tor^{ \Phi }_i( e,x )
  \rightarrow \Tor^{ \Phi }_i( e'',x ).
\]
But $\Tor^{ \Phi }_{ >\ell }( e',x ) = 0$ holds by induction, and
$\Tor^{ \Phi }_{ >\ell }( e'',x )$ is zero because it is isomorphic to
the expression in Equation \eqref{equ:Tor_vanishing} since $e'' =
\Sigma^{ -j }e^j$.  Hence the long exact sequence shows $\Tor^{ \Phi
}_{ >\ell }( e,x ) = 0$ as desired.
\end{proof}

\begin{Remark}
[Induced homomorphisms of $\Ext$ groups]
\label{rmk:Ext_map}
Let $\mod( \Gamma ) \stackrel{ G }{ \rightarrow } \mod( \Phi )$ be an
exact functor, $n',n'' \in \mod( \Gamma )$ modules, $i \geqslant 1$
an integer.
\begin{enumerate}
\setlength\itemsep{4pt}

  \item  There is an induced homomorphism of Yoneda $\Ext$ groups,
\begin{equation}
\label{equ:Ext_map}
  \xymatrix {
    \Ext_{ \Gamma }^i( n'',n' ) \ar[rr]^-{ G( - ) }
    & & \Ext_{ \Phi }^i( Gn'',Gn' ),
            }
\end{equation}
given by
\begin{align*}
  & G \big( [ 0 \rightarrow n' \rightarrow e_i \rightarrow \cdots
    \rightarrow e_1 \rightarrow n'' \rightarrow 0 ] \big) \\
  & \;\;\;\;\;\;\;\;\;\;\;\;\;\;\;\;\;\;\;\;\;\;\;\;\;\;\;\;\;\;\;\;\;\; = [ 0 \rightarrow Gn' \rightarrow Ge_i \rightarrow \cdots
    \rightarrow Ge_1 \rightarrow Gn'' \rightarrow 0 ],
\end{align*}
where square brackets denote the class of an extension in the Yoneda
$\Ext$ group.

  \item  There is an induced homomorphism of $\Hom$ spaces in the
    derived categories,
\begin{equation}
\label{equ:HomD_map}
  \xymatrix {
    \Hom_{ \Db( \mod\,\Gamma ) }( n'',\Sigma^i n' ) \ar[rr]^-{ G( - ) }
    & & \Hom_{ \Db( \mod\,\Phi ) }( Gn'',\Sigma^i Gn' ),
            }
\end{equation}
obtained from
\[
%\tag*{$\Box$}
  \xymatrix {
    \Db( \mod\,\Gamma ) \ar[rr]^-{ G( - ) }
    & & \Db( \mod\,\Phi ),
            }
\]
the canonical extension of the exact functor $G$ to a triangulated
functor. 

\end{enumerate}
It is well known that under the canonical identification of Yoneda $\Ext$ groups with $\Hom$ spaces in the derived categories, the homomorphisms \eqref{equ:Ext_map} and \eqref{equ:HomD_map} are identified.  This implies the following lemma.
\hfill $\Box$ 
\end{Remark}

\begin{Lemma}
\label{lem:Ext_map}
If the homomorphism \eqref{equ:HomD_map} of $\Hom$ spaces in the
derived categories is bijective, then so is the homomorphism
\eqref{equ:Ext_map} of Yoneda $\Ext$ groups.
\end{Lemma}

\begin{Lemma}
\label{lem:big_embedding}
Let $s \in \mod( \Phi )$ be given.  Set $\Gamma = \End_{ \Phi }( s )$ so $s$ acquires the structure ${}_{ \Gamma }s_{ \Phi }$.  Assume the following:
\begin{enumerate}
\setlength\itemsep{4pt}

  \item  $\pd( s_{ \Phi } ) < \infty$.

  \item  $\Ext_{ \Phi }^{ \geqslant 1 }( s,s ) = 0$.

\end{enumerate}

If $\pd( {}_{ \Gamma }s ) < \infty$, then there is a functor
\begin{equation}
\label{equ:LTensor}
  \xymatrix {
  \Db( \mod\,\Gamma )
    \ar[rrr]^-{ L( - ) \;=\; - \LTensor{ \Gamma } s } & & &
  \Db( \mod\,\Phi )
            }
\end{equation}
which is full and faithful.

If ${}_{ \Gamma }s$ is projective, then the functor
\begin{equation}
\label{equ:otimes}
  \xymatrix {
  \mod( \Gamma )
    \ar[rrr]^-{ G( - ) \;=\; - \underset{ \Gamma }{ \otimes } s } & & &
  \mod( \Phi )
            }
\end{equation}
satisfies:
\renewcommand{\labelenumi}{(\alph{enumi})}
\begin{enumerate}
\setlength\itemsep{4pt}

  \item  $G$ is exact, full, and faithful.

  \item  Let $n',n'' \in \mod( \Gamma )$ and an integer $i \geqslant
    1$ be given.  Then the induced homomorphism of Yoneda $\Ext$ groups
$
  \xymatrix {
    \Ext_{ \Gamma }^i( n'',n' ) \ar[r]^-{ G( - ) }
    & \Ext_{ \Phi }^i( Gn'',Gn' )
            }
$
is bijective (cf.\ Remark \ref{rmk:Ext_map}(i)).

\end{enumerate}
\renewcommand{\labelenumi}{(\roman{enumi})}
\end{Lemma}

\begin{proof}
Observe that $s$ can be viewed as a right $( \Gamma^{ \opp } \underset{ k }{ \otimes } \Phi )$-module.  It has a projective resolution $p$ over $\Gamma^{ \opp } \underset{ k }{ \otimes } \Phi$, and $p$ consists of modules which are projective when viewed as right $\Phi$-modules.  If we set $m = \pd( s_{ \Phi } )$, then the soft truncation
\[
  q
  \;\; = \;\; \cdots
  \rightarrow 0
  \rightarrow \Image \partial^p_m
  \rightarrow p_{ m-1 }
  \rightarrow \cdots
  \rightarrow p_0
  \rightarrow 0
  \rightarrow \cdots
\]
consists of modules which are projective when viewed as right $\Phi$-modules.  We can view $q$ as a complex of $\Gamma$-$\Phi$-bimodules, and there is a quasi-isomorphism $q \stackrel{ \theta }{ \longrightarrow } s$.  When viewed as a complex of right $\Phi$-modules, $q$ is a bounded projective resolution of $s_{ \Phi }$.

Now consider the functor $L$ from Equation \eqref{equ:LTensor}.  It makes sense because ${}_{ \Gamma }s$ is finite dimensional over $k$ and has finite projective dimension, whence $- \LTensor{ \Gamma } s$ maps $\Db( \mod\,\Gamma )$ to $\Db( \mod\,\Phi )$.  Likewise, $s_{ \Phi }$ is finite dimensional over $k$ and has finite projective dimension, whence $\RHom_{ \Phi }( s,- )$ maps $\Db( \mod\,\Phi )$ to $\Db( \mod\,\Gamma )$.  Hence there is an adjoint pair
\[
  \xymatrix {
  \Db( \mod\,\Gamma )
    \ar[rrr]<1ex>^-{ L( - ) \;=\; - \LTensor{ \Gamma } s } & & &
    \Db( \mod\,\Phi ).
    \ar[lll]<1ex>^-{ \RHom_{ \Phi }( s,- ) }
            }
\]
To see that $L$ is full and faithful, it is enough by the dual of \cite[thm.\ 1, p.\ 90]{MacLane} to show that the unit $\eta$ of the adjunction is a natural equivalence.  We do so by computing
\[
  g
  \stackrel{ \eta_g }{ \longrightarrow }
  \RHom_{ \Phi }( s,g \LTensor{ \Gamma } s )
\]
for each $g \in \Db( \mod\,\Gamma )$ and showing that it is an isomorphism. 

We know that $g$ has a projective resolution $r$ which is a right bounded complex consisting of finitely generated projective right $\Gamma$-modules.  There are chain maps
\begin{equation}
\label{equ:eta_theta}
  r
  \stackrel{ \zeta_r }{ \longrightarrow }
  \Hom_{ \Phi }( s,r \underset{ \Gamma }{ \otimes } s )
  \stackrel{ \theta^{ \astsmall } }{ \longrightarrow }
  \Hom_{ \Phi }( q,r \underset{ \Gamma }{ \otimes } s )
\end{equation}
where $\zeta$ is the unit of the $\otimes$--$\Hom$-adjunction.  We claim that the image of $\theta^{ \ast } \circ \zeta_r$ in $\Db( \mod\,\Gamma )$ is isomorphic to $\eta_g$ and that $\theta^{ \ast } \circ \zeta_r$ is a quasi-isomorphism; this implies that $\eta_g$ is an isomorphism as desired.

To see that the image of $\theta^{ \ast } \circ \zeta_r$ in $\Db( \mod\,\Gamma )$ is isomorphic to $\eta_g$, note that $r$ is a projective resolution of $g$ and recall that, when viewed as a complex of right $\Phi$-modules, $q$ is a projective resolution of $s$.  Hence the target of $\theta^{ \ast } \circ \zeta_r$ is $\RHom_{ \Phi }( s,g \LTensor{ \Gamma } s )$.  To see that $\theta^{ \ast } \circ \zeta_r$ is a quasi-isomorphism, note that $\zeta_r$ is an isomorphism because $r$ consists of finitely generated projective right $\Gamma$-modules while we have $\Gamma = \End_{ \Phi }( s )$. Moreover, $r \underset{ \Gamma }{ \otimes } s$ is a right bounded complex consisting of modules in $\add( s_{ \Phi } )$, and condition (ii) in the lemma implies that
\[
  \xymatrix {
  \Hom_{ \Phi }( s,\widetilde{ s } )
    \ar[rrr]^-{ \Hom_{ \Phi }( \theta,\widetilde{ s } ) } & & &
    \Hom_{ \Phi }( q,\widetilde{ s } )
            }
\]
is a quasi-isomorphism for each $\widetilde{ s } \in \add( s_{ \Phi } )$.  But $s$ and $q$ are bounded complexes so it follows from \cite[prop.\ 2.7(b)]{CFH} that 
\[
  \xymatrix {
  \Hom_{ \Phi }( s,r \underset{ \Gamma }{ \otimes } s )
    \ar[rr]^-{ \theta^{ \astsmall } } & &
    \Hom_{ \Phi }( q,r \underset{ \Gamma }{ \otimes } s )
            }
\]
is a quasi-isomorphism.

Next consider the functor $G$ from Equation \eqref{equ:otimes}.  
Since ${}_{ \Gamma }s$ is projective, $G$ is exact.  The canonical extension of $G$ to a triangulated functor is $L$.  Since $L$ is full and faithful, we learn that $G$ is full and faithful, proving (a) in the lemma.  We also learn that the map \eqref{equ:HomD_map} is bijective.  Hence the map \eqref{equ:Ext_map} is bijective by Lemma \ref{lem:Ext_map}, proving (b) in the lemma.
\end{proof}

\begin{Definition}
[Essential images]
\label{def:essim}
If $\cG \stackrel{ G }{ \longrightarrow } \cF$ is a functor then the {\em essential image} is the full subcategory
\[
\tag*{$\Box$}
  G( \cG ) = \{ f \in \cF \,|\, f \cong G( g ) \mbox{ for some $g
    \in \cG$} \}.
\]
\end{Definition}

\begin{Proposition}
\label{pro:Proto_B}
Let $( \Phi,\cF )$ and $( \Gamma,\cG )$ be $d$-homological pairs,
$\mod( \Gamma ) \stackrel{ G }{ \rightarrow } \mod( \Phi )$ a
functor.  Assume the following:
\renewcommand{\labelenumi}{(\alph{enumi})}
\begin{enumerate}
\setlength\itemsep{4pt}

  \item  $G$ is exact and restricts to a full and faithful functor
         $\cG \rightarrow \cF$.  

  \item  Let $g',g'' \in \cG$ be given.  Then the induced homomorphism
$
  \xymatrix {
    \Ext_{ \Gamma }^d( g'',g' ) \ar[r]^-{ G( - ) }
    & \Ext_{ \Phi }^d( Gg'',Gg' )
            }
$
of Yoneda $\Ext$ groups is bijective (cf.\ Remark \ref{rmk:Ext_map}(i)). 

\end{enumerate}
\renewcommand{\labelenumi}{(\roman{enumi})}
Then the essential image $G( \cG )$ is a wide subcategory of $\cF$.  
\end{Proposition}

\begin{proof}
Using that $G$ is additive, full, and faithful, it is easy to check that $G( \cG )$ is an additive sub\-ca\-te\-go\-ry of $\cF$.  We check that $G( \cG )$ has properties (i)--(iii) in Definition \ref{def:wide}.

(i):  Consider a morphism in $G( \cG )$.  Up to isomorphism, it has the
form $Gg_1 \stackrel{ G\gamma_1 }{ \longrightarrow } Gg_0$ because $G$ is full.  Since $\cG$ is $d$-abelian, we can augment $\gamma_1$
with a $d$-kernel in $\cG$ to get
\[
  0
  \longrightarrow g_{ d+1 }
  \longrightarrow \cdots
  \longrightarrow g_2
  \longrightarrow g_1
  \stackrel{ \gamma_1 }{ \longrightarrow } g_0.
\]
This is an exact sequence in $\mod( \Gamma )$ by \cite[prop.\ 2.6]{JK}.  The sequence
\[
  0
  \longrightarrow Gg_{ d+1 }
  \longrightarrow \cdots
  \longrightarrow Gg_2
  \longrightarrow Gg_1
  \stackrel{ G\gamma_1 }{ \longrightarrow } Gg_0
\]
in $\mod( \Phi )$ is exact and has terms in $\cF$ by condition (a).  The dual of \cite[prop.\ 3.18]{J} implies that given such an exact sequence with terms in $\cF$, a $d$-kernel in $\cF$ of $G\gamma_1$ is given by $Gg_{ d+1 } \longrightarrow \cdots \longrightarrow Gg_2 \longrightarrow Gg_1$.  In particular, $G\gamma_1$ has a $d$-kernel in $\cF$ with terms which are in $G( \cG )$.

(ii):  Dual to (i). 

(iii):  Consider a $d$-exact sequence in $\cF$,
\[
  0
  \rightarrow Gg'
  \rightarrow f_d
  \rightarrow \cdots
  \rightarrow f_1
  \rightarrow Gg''
  \rightarrow 0,
\]
with $g', g'' \in \cG$.  It is a $d$-extension in $\mod( \Phi )$ so by
condition (b) it is Yoneda equivalent to
\[
  0
  \rightarrow Gg'
  \rightarrow Ge_d
  \rightarrow \cdots
  \rightarrow Ge_1
  \rightarrow Gg''
  \rightarrow 0
\]
for some $d$-extension
$0 \rightarrow g' \rightarrow e_d \rightarrow \cdots \rightarrow e_1
\rightarrow g'' \rightarrow 0$ in $\mod( \Gamma )$.  However, \cite[prop.\
A.1]{I3} implies that each $d$-extension between objects in the $d$-cluster tilting subcategory $\cG$ is Yoneda equivalent to a $d$-extension with each term in $\cG$.  Hence we can assume $e_d, \ldots, e_1 \in \cG$ whence $Ge_d, \ldots, Ge_1 \in G( \cG )$ as desired.
\end{proof}

{\em Proof }(of Theorem B).
The $d$-homological pair $( \Phi,\cF )$ is given, setting $\cG = \Hom_{ \Phi }( s,\cW )$ gives a $d$-homological pair $( \Gamma,\cG )$ by condition (iv) in Theorem B, and we set $G = - \underset{ \Gamma }{ \otimes } s : \mod( \Gamma ) \rightarrow
\mod( \Phi )$.

Given $w \in \cW$, consider $g = \Hom_{ \Phi }( s,w )$, which is in $\cG$ by definition.  Condition (iii) in Theorem B gives an exact sequence
\[
  \cdots \to 0 \to p_m \to \cdots \to p_1 \to p_0 \to w \to 0.
\]
Applying $\Hom_\Phi(s,-)$ gives
\[
  \cdots \to 0 \to q_m \to \cdots \to q_1 \to q_0 \to g \to 0.
\]
This is an exact sequence in $\mod( \Gamma )$ since $p_i \in \add( s_{ \Phi } )$ and $\Ext_{ \Phi }^{ \geqslant 1 }( s,s ) = 0$.  Hence it is an augmented projective resolution of $g$. Applying $- \underset{ \Gamma }{ \otimes
} s$ to
\[
  0 \to q_m \to \cdots \to q_1 \to q_0 \to 0,
\]
we obtain a complex isomorphic to
\[
  0 \to p_m \to \cdots \to p_1 \to p_0 \to 0.
\]
Indeed, this follows as $\Hom_{ \Phi }( s,- )$ and $- \underset{
  \Gamma }{ \otimes } s$ restrict to quasi-inverse equivalences
between $\add( s_{ \Phi } )$ and $\add( \Gamma_{ \Gamma } )$.  Consequently, 
\[
  \Tor^{ \Gamma }_i( g,s )
  = 
  \left\{
    \begin{array}{ll}
      w & \mbox{for } i = 0, \\[1mm]
      0 & \mbox{for } i > 0.
    \end{array}
  \right.
\]
Observe that if $w$ varies through $\cW$, then $g = \Hom_{ \Phi }( s,w )$ varies through all of $\cG$ by definition.  Hence we conclude $\cG \underset{ \Gamma }{ \otimes } s = \cW$, that is
$\cW = G( \cG )$.  We also conclude $\Tor^{ \Gamma }_{ >0 }( \cG,s ) = 0$ whence ${}_{\Gamma}s$ is
projective by Lemma \ref{lem:homdim}.  Combining with conditions (i) and (ii) in Theorem B shows that Lemma \ref{lem:big_embedding} applies.  Hence parts (a) and (b) of that lemma say that parts (a) and (b) of Proposition \ref{pro:Proto_B} hold, so that proposition implies that $G( \cG )$ is wide.  Lemma \ref{lem:big_embedding}(a) gives that $G$ is full and faithful, so it restricts to an equivalence
\[
\tag*{$\Box$}
  - \underset{ \Gamma }{ \otimes } s : \cG \rightarrow \cW.
\]

\section{Properties of wide subcategories of $d$-abelian categories}
\label{sec:wide}

This section shows some properties of wide subcategories of
$d$-abelian categories in general and of $d$-cluster tilting
subcategories in particular.

\begin{Proposition}
\label{pro:mono_epi}
Let $\cW$ be a wide subcategory of a $d$-abelian category $\cM$, and
let $w_1 \stackrel{ \varphi }{ \longrightarrow } w_0$ be a morphism in
$\cW$.
\begin{enumerate}
\setlength\itemsep{4pt}

  \item  $\varphi$ is a monomorphism in $\cW$
    $\Leftrightarrow$ $\varphi$ is a monomorphism in $\cM$.

  \item  $\varphi$ is an epimorphism in $\cW$
    $\Leftrightarrow$ $\varphi$ is an epimorphism in $\cM$.

\end{enumerate}
\end{Proposition}

\begin{proof}
(i)  The implication $\Leftarrow$ is clear.  To show $\Rightarrow$,
assume that $\varphi$ is not a monomorphism in $\cM$.
Then there is a non-zero morphism $m \stackrel{ \mu }{ \longrightarrow
} w_1$ in $\cM$ such that $\mu \neq 0$ but $\varphi\mu = 0$.  By
Definition \ref{def:wide}(i) we can augment $\varphi$ with an
$d$-kernel in $\cM$,
\[
  0
  \longrightarrow w_{ d+1 }
  \longrightarrow \cdots
  \longrightarrow w_2
  \stackrel{ \psi }{ \longrightarrow } w_1
  \stackrel{ \varphi }{ \longrightarrow } w_0,
\]
with $w_i \in \cW$ for each $i$.  Since $\varphi\mu = 0$, there is a
morphism $m \stackrel{ \nu }{ \longrightarrow } w_2$ with $\psi\nu =
\mu$.  Then $\mu \neq 0$ implies $\psi \neq 0$, but we have
$\varphi\psi = 0$ so $\varphi$ is not a monomorphism in $\cW$.

(ii)  is dual to (i).
\end{proof}

\begin{Proposition}
\label{pro:W_is_n-abelian}
Let $\cW$ be a wide subcategory of a $d$-abelian category $\cM$.
\begin{enumerate}
\setlength\itemsep{4pt}

  \item  $\cW$ is a $d$-abelian category.  

  \item  Each $d$-exact sequence in the $d$-abelian category $\cW$ is
    a $d$-exact sequence when viewed in $\cM$.

\end{enumerate}
\end{Proposition}

\begin{proof}
(i):  By definition, $\cW$ is an additive subcategory of $\cM$, that is, it is full and closed under sums and summands in $\cM$.  In particular, it is an
additive category.  We show that $\cW$ satisfies the axioms for
$d$-abelian categories, see Definition \ref{def:d-abelian}.

(A0):  $\cW$ is idempotent complete because it is closed under
summands in the idempotent complete category $\cM$.

(A1):  It is clear from Definition \ref{def:wide}, parts (i) and (ii),
that each morphism in $\cW$ has a $d$-kernel and a $d$-cokernel in
$\cW$.

(A2'):  Let $w_{ d+1 } \stackrel{ \varphi }{ \longrightarrow } w_d$ be
a monomorphism in $\cW$.  By Definition \ref{def:wide}(ii) we can
complete with a $d$-cokernel in $\cM$ which consists of objects from
$\cW$ to get the following diagram.
\begin{equation}
\label{equ:A2}
  0
  \longrightarrow w_{ d+1 }
  \stackrel{ \varphi }{ \longrightarrow } w_d
  \longrightarrow w_{ d-1 }
  \longrightarrow \cdots
  \longrightarrow w_0
  \longrightarrow 0
\end{equation}
Proposition \ref{pro:mono_epi}(i) says that $\varphi$ is a
monomorphism in $\cM$, so axiom (A2) for $\cM$ gives that
\eqref{equ:A2} is a $d$-exact sequence in $\cM$.  This clearly implies
that it is $d$-exact in $\cW$, verifying axiom (A2') for $\cW$.

(A2${}^{\opp}$'):  Dual to (A2').  

(ii):  Suppose that \eqref{equ:A2} is a $d$-exact sequence in $\cW$
which is a $d$-abelian category by (i).  Then $\varphi$ is a
monomorphism in $\cW$, so also a monomorphism in $\cM$ by Proposition
\ref{pro:mono_epi}(i).  Completing it with a $d$-cokernel in $\cM$
gives an $d$-exact sequence in $\cM$,
\begin{equation}
\label{equ:A2b}
  0
  \longrightarrow w_{ d+1 }
  \stackrel{ \varphi }{ \longrightarrow } w_d
  \longrightarrow w_{ d-1 }'
  \longrightarrow \cdots
  \longrightarrow w_0'
  \longrightarrow 0,
\end{equation}
see \cite[def.\ 3.1, axiom (A2)]{J}.  By Definition
\ref{def:wide}(ii) we can assume $w_{ d-1 }', \ldots, w_0' \in \cW$,
and then \eqref{equ:A2b} is clearly a $d$-exact sequence in $\cW$.

It follows from \cite[prop.\ 2.7]{J} that \eqref{equ:A2} and
\eqref{equ:A2b} are homotopy equivalent.  They remain so after
applying $\cM( m,- )$, so
\begin{equation}
\label{equ:M_of_A2}
  0
  \longrightarrow \cM( m,w_{ d+1 } )
  \stackrel{ \varphi_{ \astsmall } }{ \longrightarrow } \cM( m,w_d )
  \longrightarrow \cM( m,w_{ d-1 } )
  \longrightarrow \cdots
  \longrightarrow \cM( m,w_0 )
\end{equation}
is homotopy equivalent to
\begin{equation}
\label{equ:M_of_A2b}
  0
  \longrightarrow \cM( m,w_{ d+1 } )
  \stackrel{ \varphi_{ \astsmall } }{ \longrightarrow } \cM( m,w_d )
  \longrightarrow \cM( m,w_{ d-1 }' )
  \longrightarrow \cdots
  \longrightarrow \cM( m,w_0' ).
\end{equation}
But \eqref{equ:A2b} is $d$-exact in $\cM$, so \eqref{equ:M_of_A2b} is
exact for each $m \in \cM$, hence \eqref{equ:M_of_A2} is exact for
each $m \in \cM$.  Similarly,
\[
  0
  \longrightarrow \cM( w_0,m )
  \longrightarrow \cdots
  \longrightarrow \cM( w_{ d-1 },m )
  \longrightarrow \cM( w_d,m )
  \stackrel{ \varphi^{ \astsmall } }{ \longrightarrow } \cM( w_{ d+1 },m )
\]
is exact for each $m \in \cM$.  This shows that \eqref{equ:A2} is
$d$-exact in $\cM$. 
\end{proof}

\begin{Definition}
[Envelopes, see {\cite[sec.\ 1]{E}}]
\label{def:env}
Let $\cF$ be a category and $\cW \subseteq \cF$ a full sub\-ca\-te\-go\-ry.  A morphism $f \stackrel{ \varphi }{ \rightarrow } w$ in $\cF$ with $w \in \cW$ is called:
\begin{itemize}
\setlength\itemsep{6pt}

  \item  A {\em $\cW$-preenvelope of $f$} if it has the extension property 
\[
\vcenter{
  \xymatrix {
    f \ar[r]^-{ \varphi } \ar[d] & w \ar@{.>}^<<<<{\exists}[ld] \\
    w'
            }
        }
\]
for each morphism $f \rightarrow w'$ with $w' \in \cW$.

  \item  A {\em $\cW$-envelope of $f$} if it is a $\cW$-preenvelope such that
  each morphism $w \stackrel{ \omega }{ \rightarrow } w$
  with $\omega\varphi = \varphi$ is an automorphism.

  \item  A {\em strong $\cW$-envelope of $f$} if it has the unique
    extension property 
\[
\vcenter{
  \xymatrix {
    f \ar[r]^-{ \varphi } \ar[d] & w \ar@{.>}^<<<<{\exists!}[ld] \\
    w'
            }
        }
\]
for each morphism $f \rightarrow w'$ with $w' \in \cW$.

\end{itemize}
A strong envelope is clearly an envelope.  We say that $\cW$ is {\em
  preenveloping} (resp.\ {\em enveloping}, resp.\ {\em strongly
  enveloping}) in $\cF$ if each $f \in \cF$ has a $\cW$-preenvelope
(resp.\ envelope, resp.\ strong envelope).

The dual notion to preenvelope is {\em precover}, and $\cW$ is called
{\em functorially finite} if it is preenveloping and precovering.
\hfill $\Box$
\end{Definition}

\begin{Lemma}
\label{lem:env}
Let $\cF$ be a category and let $\cW \subseteq \cF$ be a full subcategory.  The inclusion functor $\cW \stackrel{ i_{ \astsmall } }{ \longrightarrow } \cF$ has a left adjoint $\cF \stackrel{ i^{ \astsmall } }{ \longrightarrow } \cW$ if and only if $\cW$ is strongly enveloping in $\cF$.

In this case, each $f \in \cF$ has the strong $\cW$-envelope $f \stackrel{ \eta_f }{ \longrightarrow } i_{ \ast }i^{ \ast }f$ where $\eta$ is the unit of the adjunction.
\end{Lemma}

\begin{proof}
If $i^{ \ast }$ exists, then adjoint functor theory implies that the
unit of the adjunction $\eta$ gives a strong $\cW$-envelope
$f \stackrel{ \eta_f }{ \longrightarrow } i_{ \ast }i^{ \ast }f$ for
each $f \in \cF$, see \cite[p.\ 91, line 10 from the bottom]{MacLane}.
Conversely, if $\cW$ is strongly enveloping in $\cF$, then we can get
a left adjoint $i^{ \ast }$ by setting $i^{ \ast }f = w$ when
$f \rightarrow w$ is a strong $\cW$-envelope.
\end{proof}

\begin{Proposition}
\label{pro:envelopes_and_adjoints}
Let $\cW$ be a wide subcategory of a $d$-abelian category $\cM$.
Then:
\begin{enumerate}
\setlength\itemsep{4pt}

  \item  The inclusion functor
    $\cW \stackrel{ i_{ \astsmall } }{ \longrightarrow } \cM$ has a left adjoint
    $\cM \stackrel{ i^{ \astsmall } }{ \longrightarrow } \cW$ if and only if $\cW$ is 
    enveloping in $\cM$.

  \item  Each $\cW$-envelope in $\cM$ is a strong $\cW$-envelope in
    $\cM$ (see Definition \ref{def:env}).  

\end{enumerate}
\end{Proposition}

\begin{proof}
(ii):  Let $m \stackrel{ \mu }{ \longrightarrow } w_1$ be a $\cW$-envelope.  In particular, each morphism $m \rightarrow w_0$ with $w_0 \in \cW$ factors through $\mu$.  To show that $\mu$ is a strong $\cW$-envelope, we must show that the factorisation is unique.  That is, if $w_1 \stackrel{ \omega_1 }{ \longrightarrow } w_0$ satisfies $\omega_1\mu = 0$, then we must show $\omega_1 = 0$.  

Since $\cW$ is wide, we can augment $\omega_1$ with a $d$-kernel in $\cM$ which consists of objects
from $\cW$:
\[
  0
  \longrightarrow w_{ d+1 }
  \longrightarrow \cdots 
  \longrightarrow w_2
  \stackrel{ \omega_2 }{ \longrightarrow } w_1
  \stackrel{ \omega_1 }{ \longrightarrow } w_0.
\]
Observe that $\omega_1\omega_2 = 0$.  Since $\omega_1\mu = 0$, there is
$m \stackrel{ \widetilde{ \mu } }{ \longrightarrow } w_2$ with
$\omega_2\widetilde{ \mu } = \mu$.  Since $w_2 \in \cW$, there is
$w_1 \stackrel{ \widetilde{ \omega } }{ \longrightarrow } w_2$ with
$\widetilde{ \omega }\mu = \widetilde{ \mu }$.  Hence
$\omega_2\widetilde{ \omega }\mu = \omega_2\widetilde{ \mu } = \mu$.
Since $\mu$ is a $\cW$-envelope, it follows that
$\omega_2\widetilde{ \omega }$ is an automorphism.  But then
$0 = 0 \circ \widetilde{ \omega } = \omega_1\omega_2 \circ \widetilde{ \omega }
= \omega_1 \circ \omega_2\widetilde{ \omega }$ implies $\omega_1 = 0$.

(i): Follows from (ii) and Lemma \ref{lem:env}.
\end{proof}

\begin{Definition}
[Three types of projective objects]
\label{def:projectives}
Let $\cW$ be an additive category.
\begin{itemize}
\setlength\itemsep{6pt}

  \item  A {\em categorically projective object of $\cW$} is an object $p \in \cW$ such that each epimorphism $w_1 \rightarrow w_0$ in $\cW$
induces a surjection $\cW( p,w_1 ) \rightarrow \cW( p,w_0 )$.

\end{itemize}
Now let $\cW \subseteq \mod( \Phi )$ be an additive subcategory.
\begin{itemize}
\setlength\itemsep{6pt}

  \item  A {\em splitting projective object} of $\cW$ is an object $p_0 \in \cW$ such that each surjection $w \rightarrow p_0$ with $w \in \cW$ is split.

  \item  An {\em $\Ext^d$-projective object} of $\cW$ is an object
    $p_e \in \cW$ such that $\Ext_{ \Phi }^d( p_e,\cW ) = 0$.

\end{itemize}
The additive subcategories of categorically projective (resp.\
splitting projective, resp.\ $\Ext^d$-projective) objects of
$\cW$ are denoted $P( \cW )$ (resp.\ $P_0( \cW )$, resp.\ $P_e( \cW
)$).
\hfill $\Box$
\end{Definition}

\begin{Proposition}
\label{pro:projectives}
Let $( \Phi,\cF )$ be a $d$-homological pair and $\cW \subseteq \cF$ a
wide subcategory.  Then $P( \cW ) = P_0( \cW ) = P_e( \cW )$.
\end{Proposition}

\begin{proof}
$P( \cW ) \subseteq P_0( \cW )$: This is clear from Definition
\ref{def:projectives}.

$P_0( \cW ) \subseteq P_e( \cW )$: Let $p_0 \in P_0( \cW )$ and $w \in
\cW$ be given and consider a $d$-extension
\[
  0
  \longrightarrow w
  \longrightarrow w_d
  \longrightarrow \cdots
  \longrightarrow w_1
  \stackrel{ \omega_1 }{ \longrightarrow } p_0
  \longrightarrow 0
\]
representing an element $\varepsilon \in \Ext_{ \Phi }^d( p_0,w )$.
By \cite[prop.\ A.1]{I3} we can assume $w_i \in \cF$ for
each $i$.  Hence by Definition \ref{def:wide}(iii) we can assume $w_i
\in \cW$ for each $i$.  But $w_1 \in \cW$ and $p_0 \in P_0( \cW )$
imply that $\omega_1$ is split whence $\varepsilon = 0$, and $\Ext_{
  \Phi }^d( p_0,\cW ) = 0$ follows.

$P_e( \cW ) \subseteq P( \cW )$: Let $p_e \in P_e( \cW )$ be given.
Let $w_1 \stackrel{ \omega_1 }{ \longrightarrow } w_0$ be an
epimorphism in $\cW$.  Proposition \ref{pro:W_is_n-abelian}(i) says
that $\cW$ is $d$-abelian, so we can complete with an $d$-kernel in $\cW$ to get a $d$-exact sequence in $\cW$,
\[
  0
  \longrightarrow w_{ d+1 }
  \longrightarrow w_d
  \longrightarrow \cdots
  \longrightarrow w_2
  \longrightarrow w_1
  \stackrel{ \omega_1 }{ \longrightarrow } w_0
  \longrightarrow 0.
\]
By Proposition \ref{pro:W_is_n-abelian}(ii) this is a $d$-exact
sequence in $\cF$, so it is an exact sequence in $\mod( \Phi )$ by
\cite[lem.\ 6.1]{Jo}.  Hence \cite[prop.\ 2.2]{JK} gives a long exact
sequence containing
\[
  \Hom_{ \Phi }( p_e,w_1 )
  \stackrel{ ( \omega_1 )_{ \astsmall } }{ \longrightarrow }
  \Hom_{ \Phi }( p_e,w_0 )
  \longrightarrow
  \Ext_{ \Phi }^d( p_e,w_{ d+1 } ).
\]
The last term is $0$ since $p_e \in P_e( \cW )$ and $w_{ d+1 } \in
\cW$, so $( \omega_1 )_{ \astsmall }$ is surjective which shows $p_e
\in P( \cW )$.   
\end{proof}

\begin{Definition}
[Projectively generated $d$-abelian categories, see {\cite[def.\ 3.19]{J}}]
\label{def:projectively_generated}
A $d$-abelian category $\cM$ is called {\em projectively generated} if each $m \in \cM$ permits an epimorphism $p \rightarrow m$ where $p$ is categorically projective in $\cM$.
\hfill $\Box$
\end{Definition}

\begin{Proposition}
\label{pro:P_and_s}
Let $( \Phi,\cF )$ be a $d$-homological pair and $\cW \subseteq \cF$ a
wide subcategory.  Note that $\cW$ is a $d$-abelian category by
Proposition \ref{pro:W_is_n-abelian}(i).

Assume that the inclusion functor
$\cW \stackrel{ i_{ \astsmall } }{ \longrightarrow } \cF$ has a left
adjoint $\cF \stackrel{ i^{ \astsmall } }{ \longrightarrow } \cW$.
Set $s = i^{ \ast }( \Phi_{ \Phi } )$.  Then:
\begin{enumerate}
\setlength\itemsep{4pt}

  \item  $\cW$ is projectively generated.

  \item  $P( \cW ) = \add( s )$.

\end{enumerate}
\end{Proposition}

\begin{proof}
Observe that $i^{ \ast }i_{ \ast }$ is equivalent to the identity on
$\cW$ (see \cite[p.\ 91]{MacLane}), and that
$i^{ \ast }$ preserves epimorphisms (see \cite[p.\ 119]{MacLane}) and categorically projective objects (the dual
of the proof of \cite[lem.\ 2.7]{Po} applies).  In particular,
$s = i^{ \ast }( \Phi_{ \Phi } )$ is a categorically projective object
in $\cW$.

Let $w \in \cW$ be given.  There is a surjection
$\Phi_{ \Phi }^n \rightarrow i_{ \ast }w$.  It is an epimorphism in
$\cF$, so
$i^{ \ast }( \Phi_{ \Phi }^n ) \rightarrow i^{ \ast }i_{ \ast }w$,
which is isomorphic to $s^n \rightarrow w$, is an epimorphism in $\cW$
with $s^n$ a categorically projective object in $\cW$.  This shows
(i).

For (ii) we show the two inclusions:

$\subseteq$:  Let $p \in P( \cW )$ be given.  Setting $w = p$ above
shows that there is an epimorphism $s^n \rightarrow p$.
It must split whence $p \in \add( s )$.

$\supseteq$:  Immediate because $s$ is a categorically projective
object in $\cW$.  
\end{proof}

\begin{Proposition}
\label{pro:s_resolution}
Let $( \Phi,\cF )$ be a $d$-homological pair and $\cW \subseteq \cF$ a
wide subcategory.  Note that $\cW$ is a $d$-abelian category by
Proposition \ref{pro:W_is_n-abelian}(i).

Assume that $\cW$ is projectively generated.  Then for each $w \in \cW$
there is a sequence
\begin{equation}
\label{equ:s-resolution}
  \cdots
  \rightarrow p_2
  \rightarrow p_1
  \rightarrow p_0
  \rightarrow w
  \rightarrow 0
\end{equation}
in $\mod( \Phi )$ with $p_i \in P( \cW )$ for each $i$, such that:
\begin{enumerate}
\setlength\itemsep{4pt}

  \item  The sequence is exact.

  \item  The sequence remains exact when we apply the functor $\Hom_{
      \Phi }( p,- )$ with $p \in P( \cW )$. 

\end{enumerate}
\end{Proposition}

\begin{proof}
Conditions (i) and (ii) make sense for any sequence in $\mod( \Phi )$.  
We construct \eqref{equ:s-resolution} by starting with the morphism $w 
\rightarrow 0$ and using the following repeatedly: Given a morphism $w'
\stackrel{ \omega' }{ \longrightarrow } w''$ in $\cW$, there is a
sequence
\begin{equation}
\label{equ:s-resolution_step}
  p'
  \stackrel{ \pi' }{ \longrightarrow } w'
  \stackrel{ \omega' }{ \longrightarrow } w''
\end{equation}
in $\mod( \Phi )$ with $p' \in P( \cW )$, which satisfies conditions (i) and (ii).

To show this, complete $w' \stackrel{ \omega' }{
  \longrightarrow } w''$ with a $d$-kernel in $\cF$, 
\begin{equation}
\label{equ:n-kernel3}
  0
  \longrightarrow w_d
  \longrightarrow \cdots
  \longrightarrow w_1
  \stackrel{ \omega_1 }{ \longrightarrow } w'
  \stackrel{ \omega' }{ \longrightarrow } w'',
\end{equation}
where we can assume $w_i \in \cW$ for each $i$ by Definition
\ref{def:wide}(i).  Let
$p' \stackrel{ \pi_1 }{ \longrightarrow } w_1$ be an epimorphism in
$\cW$ with $p' \in P( \cW )$.  Let $\pi'$ be the composition
$p' \stackrel{ \pi_1 }{ \longrightarrow } w_1 \stackrel{ \omega_1 }{
  \longrightarrow } w'$.  Then \eqref{equ:s-resolution_step} satisfies conditions (i) and (ii):

(ii):  By the definition of $d$-kernels, \eqref{equ:n-kernel3} becomes
exact upon applying the functor $\Hom_{ \Phi }( f,- )$ for
$f \in \cF$.  In particular, the part
$w_1 \stackrel{ \omega_1 }{ \longrightarrow } w' \stackrel{ \omega' }{
  \longrightarrow } w''$ becomes exact upon applying
$\Hom_{ \Phi }( p,- )$ for $p \in P( \cW )$.  Moreover,
$p' \stackrel{ \pi_1 }{ \longrightarrow } w_1$ is an epimorphism in
$\cW$ so becomes surjective upon applying $\Hom_{ \Phi }( p,- )$.
This proves (ii).

(i):  This follows by replacing $p$ with $\Phi$ in the proof of
(ii).  The reason $p' \stackrel{ \pi_1 }{ \longrightarrow } w_1$
becomes surjective upon applying $\Hom_{ \Phi }( \Phi,- )$ is
that $\pi_1$ is an epimorphism in $\cF$ by Proposition
\ref{pro:mono_epi}(ii). 
\end{proof}

\section{Morphisms of $d$-homological pairs}
\label{sec:d-homological_pairs}

Recall that Definition \ref{def:d-homological_pair} introduced $d$-homological pairs.  This section defines morphisms of $d$-homological pairs and shows some properties.

\begin{Definition}
[Morphisms of $d$-homological pairs]
\label{def:morphisms_of_d-homological_pairs}
A {\em morphism of $d$-homological pairs}
$( \Phi,\cF ) \stackrel{ \phi }{ \longrightarrow } ( \Gamma,\cG )$ is
an algebra homomorphism
$\Phi \stackrel{ \phi }{ \longrightarrow } \Gamma$ satisfying
$\phi_{ \ast }( \cG ) \subseteq \cF$, where
$\phi_{ \ast } : \mod( \Gamma ) \rightarrow \mod( \Phi )$ is the
functor given by restriction of scalars.  Given such a $\phi$, there
is a functor 
\[
  \cG
  \stackrel{ \phi_{ \bullet } }{ \longrightarrow }
  \cF
\]
defined by $\phi_{\bullet} = \phi_{ \ast }|_{ \cG }$.

We say that $\phi$ is:
\begin{itemize}
\setlength\itemsep{4pt}

  \item  An {\em isomorphism of $d$-homological pairs} if the algebra
    homomorphism $\phi$ is bijective.  Then $\phi_{ \astsmall }$ is an equivalence, and so is $\phi_{ \bullet }$ since one $d$-cluster tilting subcategory cannot be strictly contained in another.

  \item  An {\em epimorphism of $d$-homological pairs} if the algebra
    homomorphism $\phi$ is an epimorphism in the category of rings.

  \item  {\em $d$-pseudoflat} if $\Tor_d^{ \Phi }( \Gamma,\Gamma ) = 0$.
\end{itemize}

A second morphism of $d$-homological pairs
$( \Phi,\cF ) \stackrel{ \widetilde{ \phi } }{ \longrightarrow } (
\widetilde{ \Gamma },\widetilde{ \cG } )$ is {\em equivalent} to
$\phi$ if there is an isomorphism of $d$-homological pairs
$( \widetilde{ \Gamma },\widetilde{ \cG } ) \stackrel{ \widetilde{
    \gamma } }{ \longrightarrow } ( \Gamma,\cG )$ such that
$\widetilde{ \gamma }\widetilde{ \phi } = \phi$ as algebra
homomorphisms.  \hfill $\Box$
\end{Definition}

\begin{Remark}
\label{rmk:epimorphism}
Let $( \Phi,\cF ) \stackrel{ \phi }{ \rightarrow } ( \Gamma,\cG )$ be
a morphism of $d$-homological pairs.  Then the following are
equivalent:
\begin{enumerate}
\setlength\itemsep{4pt}

  \item  $\phi$ is an epimorphism of $d$-homological pairs.

  \item  The functor
    $\mod( \Gamma ) \stackrel{ \phi_{ \astsmall } }{ \longrightarrow } \mod( \Phi
    )$ is full and faithful.

  \item  The multiplication map $\Gamma \underset{ \Phi }{ \otimes } \Gamma 
\rightarrow \Gamma$ is bijective.

\end{enumerate}
Indeed, these properties only concern the algebra homomorphism
$\Phi \stackrel{ \phi }{ \rightarrow } \Gamma$ and the module
categories $\mod( \Phi )$ and $\mod( \Gamma )$, and they are
equivalent by \cite[prop.\ 1.1]{Si} and \cite[prop.\ 1.2]{St}.
\hfill $\Box$
\end{Remark}

\begin{Remark}
[Functors associated to morphisms]
\label{rmk:adjoints}
Let $\Phi \stackrel{ \phi }{ \rightarrow } \Gamma$ be an algebra
homomorphism.  The restriction of scalars functor $\phi_{ \ast }$ is
exact.  It satisfies
\[
  \phi_{ \ast }( - )
  \cong \Hom_{ \Gamma }( \Gamma,- )
  \cong - \underset{ \Gamma }{ \otimes } \Gamma
\]
so $\phi_{ \ast }$ has left and right adjoint functors
\[
  \phi^{ \ast }(-) = - \underset{ \Phi }{ \otimes }\Gamma
  \;\;,\;\;
  \phi^!(-) = \Hom_{ \Phi }( \Gamma,- )
\]
and we get a diagram
\[
\tag*{$\Box$}
\vcenter{
\xymatrix
{
\mod( \Gamma )
    \ar[rrr]^{ \phi_{ \astsmall }(-) } &&&
  \mod( \Phi ).
    \ar@/_2.0pc/[lll]_{ \phi^{ \astsmall }(-) }
    \ar@/^2.0pc/[lll]^{ \phi^!(-) }
}
        }
\]
\end{Remark}

\begin{Remark}
\label{rmk:RHom_Ext}
Recall that if $u$, $v$, $x$, $y$ are modules, then
\[
\tag*{$\Box$}
  \H^i\!\RHom( u,v ) \cong \Ext^i( u,v )
  \;\; \mbox{and} \;\;
  \H_i( x \LTensor{} y ) \cong \Tor_i( x,y ).
\]
\end{Remark}

\begin{Proposition}
\label{pro:adjoints}
Let $( \Phi,\cF ) \stackrel{ \phi }{ \rightarrow } ( \Gamma,\cG )$ be
a morphism of $d$-homological pairs.  Then $\phi^{ \ast }$,
$\phi_{ \ast }$, $\phi^!$ restrict to functors
\[
\xymatrix
{
\cG
    \ar[rrr]^{ \phi_{ \bullet } } &&&
  \cF
    \ar@/_1.5pc/[lll]_{ \phi^{ \bullet } } 
    \ar@/^1.5pc/[lll]^{ \phi^? }
}
\]
where each is a left adjoint of the one below it.
\end{Proposition}

\begin{proof}
We have $\phi_{ \ast }( \cG ) \subseteq \cF$ by definition, and it is clearly enough to show
\begin{equation}
\label{equ:inclusions}
  \phi^{ \ast }( \cF ) \subseteq \cG
  \;\;,\;\;
  \phi^!( \cF ) \subseteq \cG.
\end{equation}
Observe that if $f \in \cF$ and $g \in \cG$, then
$\phi_{ \ast }( g ) \in \cF$ whence
\begin{equation}
\label{equ:RHom_gf}
  \H^i\!\RHom_{ \Phi }( \phi_{ \ast }g,f ) = 0
  \;\;\mbox{for}\;\;
  i \in \{ 1, \ldots, d-1 \}
\end{equation}
by Remark \ref{rmk:RHom_Ext}.

The second inclusion in \eqref{equ:inclusions} can be proved as
follows: Let $f \in \cF$ be given.  Then Remark \ref{rmk:RHom_Ext} implies
$\H^i\!\RHom_{ \Phi }( \Gamma,f ) = 0$ for $i \leqslant -1$ and
$\H^0\!\RHom_{ \Phi }( \Gamma,f ) \cong \Hom_{ \Phi }( \Gamma,f )$, so
there is a triangle
\begin{equation}
\label{equ:RHom_triangle}
  \Hom_{ \Phi }( \Gamma,f )
  \rightarrow \RHom_{ \Phi }( \Gamma,f )
  \rightarrow x
  \rightarrow 
\end{equation}
in $\cD( \Gamma )$ where $\H^i\!x = 0$ for $i \leqslant 1$.  In fact, we even have
\begin{equation}
\label{equ:H_concentrated}
  \H^i\!x = 0
  \;\;\mbox{for}\;\;
  i \leqslant d-1.
\end{equation}
To see this, note that $\H^i\!\RHom_{ \Phi }( \Gamma,f ) = 0$ for
$i = 1, \ldots, d-1$ by \eqref{equ:RHom_gf} with
$g = \Gamma_{ \Gamma }$ and use the long exact cohomology sequence
of \eqref{equ:RHom_triangle}.

Now let $g \in \cG$ be given.  Applying $\RHom_{ \Gamma }( g,- )$ to
\eqref{equ:RHom_triangle} gives a triangle in $\cD( k )$,
\begin{equation}
\label{equ:RHom_RHom_triangle}
  \RHom_{ \Gamma }\!\big( g,\Hom_{ \Phi }( \Gamma,f ) \big)
  \rightarrow \RHom_{ \Gamma }\!\big( g,\RHom_{ \Phi }( \Gamma,f ) \big)
  \rightarrow \RHom_{ \Gamma }( g,x )
  \rightarrow,
\end{equation}
where the first term is
$\RHom_{ \Gamma }( g,\phi^!f )$.  The second
term can be rewritten using adjointness of $\LTensor{ \Gamma }$ and
$\RHom_{ \Phi }$:
\[
  \RHom_{ \Gamma }\!\big( g,\RHom_{ \Phi }( \Gamma,f ) \big)
  \cong \RHom_{ \Phi }( g \LTensor{ \Gamma } \Gamma,f )
  \cong \RHom_{ \Phi }( \phi_{ \ast }g,f ).
\]
Hence \eqref{equ:RHom_RHom_triangle} is isomorphic to the triangle
\begin{equation}
\label{equ:RHom_RHom_triangle2}
  \RHom_{ \Gamma }( g,\phi^!f )
  \rightarrow \RHom_{ \Phi }( \phi_{ \ast }g,f )
  \rightarrow \RHom_{ \Gamma }( g,x )
  \rightarrow.
\end{equation}

The cohomology of the second term satisfies \eqref{equ:RHom_gf}.  The
cohomology of the third term satisfies
$\H^i\! \RHom_{ \Gamma }( g,x ) = 0$ for $i \leqslant d-1$ because of
\eqref{equ:H_concentrated}.  Hence the long exact sequence of
\eqref{equ:RHom_RHom_triangle2} shows that the cohomology of the first
term satisfies
\[
  \H^i\!\RHom_{ \Gamma }( g,\phi^!f ) = 0
  \;\;\mbox{for}\;\;
  i = 1, \ldots, d-1.
\]
By Remark \ref{rmk:RHom_Ext} this shows $\phi^!f \in \cG$ as
desired, since $g \in \cG$ is arbitrary.

A variation of this method proves the first inclusion in
\eqref{equ:inclusions}. 
\end{proof}

\begin{Proposition}
\label{pro:epimorphism}
A morphism of $d$-homological pairs
$( \Phi,\cF ) \stackrel{ \phi }{ \longrightarrow } ( \Gamma,\cG )$ is an
epimorphism of $d$-homological pairs if and only if the functor
$\cG \stackrel{ \phi_{ \bullet } }{ \longrightarrow } \cF$ is full and
faithful.
\end{Proposition}

\begin{proof}
By Remark \ref{rmk:epimorphism}, we must show that
$\mod( \Gamma ) \stackrel{ \phi_{ \astsmall } }{ \longrightarrow } \mod( \Phi )$
is full and faithful if and only if so is
$\cG \stackrel{ \phi_{ \bullet } }{ \longrightarrow } \cF$.

On the one hand, if $\phi_{ \ast }$ is full and faithful, then so is
its restriction $\phi_{ \bullet }$.  On the other hand, assume that
$\phi_{ \bullet }$ is full and faithful.  Since $\phi_{ \ast }$ is
restriction of scalars, it is clearly faithful.  To see that it is
full, let a morphism
$\phi_{ \ast }( n ) \stackrel{ \psi }{ \rightarrow } \phi_{ \ast }( n'
)$ be given.  By Lemma \ref{lem:resolution} there are augmented
$\cG$-resolutions
\[
  \cdots \rightarrow 0 \rightarrow g_{ d-1 } \rightarrow \cdots \rightarrow g_0
  \rightarrow n \rightarrow 0
  \;\;,\;\;
  \cdots \rightarrow 0 \rightarrow g'_{ d-1 } \rightarrow \cdots \rightarrow g'_0
  \rightarrow n' \rightarrow 0.
\]
By Remark \ref{rmk:adjoints}, the functor $\phi_{ \ast }$ sends them to the exact sequences which form the top and bottom rows of the following diagram.
\[
\vcenter{
  \xymatrix {
    \cdots \ar[r]
      & 0 \ar[r] \ar[d]
      & \phi_{ \bullet }( g_{ d-1 } ) \ar[r] \ar^{ \psi_{ d-1 } }[d] 
      & \cdots \ar[r]
      & \phi_{ \bullet }( g_0 ) \ar[r] \ar^{ \psi_0 }[d] 
      & \phi_{ \ast }( n ) \ar[r] \ar[d]^{ \psi } & 0 \\ 
    \cdots \ar[r]
      & 0 \ar[r]
      & \phi_{ \bullet }( g'_{ d-1 } ) \ar[r]
      & \cdots \ar[r]
      & \phi_{ \bullet }( g'_0 ) \ar[r] 
      & \phi_{ \ast }( n' ) \ar[r] 
      & 0 \\ 
            }
        }
\]
Here $\psi$ lifts to the $\psi_i$ because the
$\phi_{ \bullet }( g_i )$ and $\phi_{ \bullet }( g'_i )$ are in $\cF$ by assumption, whence the top and bottom rows are augmented left $\cF$-resolutions of $\phi_{ \ast }( n )$ and $\phi_{ \ast }( n' )$ by Lemma \ref{lem:resolution}.  
We have $\psi_i = \phi_{ \bullet }( \gamma_i )$ for certain $\gamma_i$ because $\phi_{ \bullet }$ is full.  This gives a commutative diagram
\[
\vcenter{
  \xymatrix {
    \cdots \ar[r]
      & 0 \ar[r] \ar[d]
      & g_{ d-1 } \ar[r] \ar^{ \gamma_{ d-1 } }[d] 
      & \cdots \ar[r]
      & g_0 \ar[r] \ar^{ \gamma_0 }[d] 
      & n \ar[r] \ar[d]^{ \gamma } & 0 \\ 
    \cdots \ar[r]
      & 0 \ar[r]
      & g'_{ d-1 } \ar[r]
      & \cdots \ar[r]
      & g'_0 \ar[r] 
      & n' \ar[r] 
      & 0 \\ 
            }
        }
\]
where the $\gamma_i$ induce $\gamma$.  It follows that $\psi =
\varphi_{ \ast }( \gamma )$.
\end{proof}

\begin{Lemma}
\label{lem:intermediate_Tor_vanishing}
Let $( \Phi,\cF ) \stackrel{ \phi }{ \rightarrow } ( \Gamma,\cG )$ be
a morphism of $d$-homological pairs.  Then
$\Tor_i^{ \Phi }( \Gamma,\Gamma ) = 0$ for
$i \in \{ 1, \ldots, d-1 \}$.
\end{Lemma}

\begin{proof}
There is an isomorphism
$\Tor_i^{ \Phi }( \Gamma,\Gamma ) \cong \dual\!\Ext^i_{ \Phi }(
\Gamma,\dual\!\Gamma ) = ( \ast )$, and we can write $( \ast )$ more
elaborately as
\[
  ( \ast ) 
  = \dual\!\Ext_{ \Phi }^i\!
      \big( 
        \phi_{ \ast }( \Gamma_{ \Gamma } ),
        \phi_{ \ast }( ( \dual\!\Gamma )_{ \Gamma } )
      \big).
\]
This expression is zero for $i = 1, \ldots, d-1$ because
$\Gamma_{ \Gamma }$ and $( \dual\!\Gamma )_{ \Gamma }$ are projective
and injective, hence in $\cG$, whence
$\phi_{ \ast }( \Gamma_{ \Gamma } ), \phi_{ \ast }( ( \dual\!\Gamma
)_{ \Gamma } ) \in \cF$.
\end{proof}

It is natural to include the following proposition although it will not be used later.

\begin{Proposition}
Let $( \Phi,\cF ) \stackrel{ \phi }{ \rightarrow } ( \Gamma,\cG )$ be
a $d$-pseudoflat epimorphism of $d$-homological pairs.
\begin{enumerate}
\setlength\itemsep{4pt}

  \item  If $\gldim( \Phi ) \leqslant d$ then $\gldim( \Gamma ) \leqslant d$ and $\phi$ is a homological epimorphism in the sense of \cite[def.\ 4.5]{GL}.
  
  \item  If $\Phi$ is $d$-representation finite in the sense of \cite[def.\ 2.1]{IO}, then so is $\Gamma$.

\end{enumerate}
\end{Proposition}

\begin{proof}
(i):  It is enough to show that $\phi$ is a homological epimorphism since $\gldim( \Gamma ) \leqslant d$ then follows from \cite[cor.\ 4.6]{GL}.

The multiplication map $\Gamma \underset{ \Phi }{ \otimes } \Gamma \rightarrow \Gamma$ is bijective by Remark \ref{rmk:epimorphism}.  We have $\Tor_i^{ \Phi }( \Gamma,\Gamma ) = 0$ for $1 \leqslant i$: For $1 \leqslant i \leqslant d-1$ this holds by Lemma \ref{lem:intermediate_Tor_vanishing}, for $i = d$ by Definition \ref{def:morphisms_of_d-homological_pairs}, and for $d+1 \leqslant i$ by $\gldim( \Phi ) \leqslant d$.  Hence $\phi$ is a homological epimorphism.

(ii):  Suppose that $\Phi$ is $d$-representation finite.  Then $\gldim( \Phi ) \leqslant d$ by definition so $\gldim( \Gamma ) \leqslant d$ by (i).  Moreover, the $d$-cluster tilting subcategory $\cF$ is unique and has finitely many isomorphism classes of indecomposable objects by \cite[thm.\ 1.6]{I1}.  Since $\cG \stackrel{ \phi_{ \bullet } }{ \longrightarrow } \cF$ is full and faithful by Proposition \ref{pro:epimorphism}, the cluster tilting subcategory $\cG$ also has finitely many isomorphism classes of indecomposable objects.  Hence $\cG = \add( g )$ for some $g \in \mod( \Gamma )$ so $g$ is a $d$-cluster tilting object in the sense of \cite[def.\ 2.1]{IO}.  This shows that $\Gamma$ is $d$-representation finite.
\end{proof}

The rest of this section deals with the more complicated situation where $\gldim( \Phi )$ is unrestricted.  This is needed in the proof of Theorem A.

\begin{Lemma}
\label{lem:Tor_triangle}
Let $( \Phi,\cF ) \stackrel{ \phi }{ \rightarrow } ( \Gamma,\cG )$ be
a $d$-pseudoflat epimorphism of $d$-homological pairs.  Consider the
adjoint functors
\[
  \xymatrix {
  \cD( \Phi )
    \ar[rrr]<1ex>^-{ - \LTensor{ \Phi } \Gamma } & & &
    \cD( \Gamma ).
    \ar[lll]<1ex>^-{ \RHom_{ \Gamma }( \Gamma,- ) }
            }
\]
Given $n \in \mod( \Gamma )$, the counit morphism $\varepsilon_n$ can
be completed to a triangle 
\[
  x
  \longrightarrow
  \RHom_{ \Gamma }( \Gamma,n )
    \LTensor{ \Phi } \Gamma
  \stackrel{ \varepsilon_n }{ \longrightarrow }
  n
  \longrightarrow
\]
in $\cD( \Gamma )$ which satisfies
\begin{equation}
\label{equ:H_bound}
  \H_i( x ) = 0  \;\;\mbox{for} \;\; i \leqslant d.
\end{equation}
\end{Lemma}

\begin{proof}
First observe that if $n$ is a $\Gamma$-$\Gamma$-bimodule, then the triangle in the lemma lifts canonically to a triangle in $\cD( \Gamma^e )$, the derived
category of $\Gamma$-bimodules.  In the special case $n = \Gamma$ this gives a triangle
\begin{equation}
\label{equ:universal_triangle}
  c
  \longrightarrow
  \RHom_{ \Gamma }( \Gamma,\Gamma )
    \LTensor{ \Phi } \Gamma
  \stackrel{ \mu }{ \longrightarrow }
  \Gamma
  \longrightarrow
\end{equation}
in $\cD( \Gamma^e )$. 

The second term of \eqref{equ:universal_triangle} is isomorphic to
$\Gamma \LTensor{ \Phi } \Gamma$.  It follows that $\H_i$ of the
second term is zero for $i \leqslant -1$ by Remark \ref{rmk:RHom_Ext},
zero for $i = 1, \ldots, d - 1$ by Remark \ref{rmk:RHom_Ext} and Lemma
\ref{lem:intermediate_Tor_vanishing}, and zero for $i = d$ by Remark
\ref{rmk:RHom_Ext} and Definition
\ref{def:morphisms_of_d-homological_pairs}.  Moreover, $\H_0( \mu )$
is bijective since $\mu$ is isomorphic to the derived multiplication
morphism $\Gamma \LTensor{ \Phi } \Gamma \rightarrow \Gamma$, whence
$\H_0( \mu )$ can be identified with the usual multiplication morphism
$\Gamma \underset{ \Phi }{ \otimes } \Gamma \rightarrow \Gamma$ which
is bijective by Remark \ref{rmk:epimorphism}.  Hence the long exact
homology sequence of \eqref{equ:universal_triangle} shows
\begin{equation}
\label{equ:universal_H_bound}
  \H_i( c ) = 0  \;\;\mbox{for} \;\; i \leqslant d.
\end{equation}

Applying the functor $n \LTensor{ \Gamma } -$ to
\eqref{equ:universal_triangle} gives a triangle
\[
  n \LTensor{ \Gamma } c 
  \longrightarrow
  n \LTensor{ \Gamma } 
    \RHom_{ \Gamma }( \Gamma,\Gamma )
    \LTensor{ \Phi } \Gamma
  \stackrel{ n \LTensor{ \Gamma } \mu }{ \longrightarrow }
  n \LTensor{ \Gamma } \Gamma
  \longrightarrow
\]
in $\cD( \Gamma )$ which can be identified with
\[
  n \LTensor{ \Gamma } c 
  \longrightarrow
    \RHom_{ \Gamma }( \Gamma,n )
    \LTensor{ \Phi } \Gamma
  \stackrel{ \varepsilon_n }{ \longrightarrow }
  n
  \longrightarrow.
\]
Comparing with the triangle in the lemma shows $x \cong n \LTensor{ \Gamma } c$ whence \eqref{equ:universal_H_bound} implies \eqref{equ:H_bound}.
\end{proof}

\begin{Lemma}
\label{lem:Ext_isomorphism}
Let $( \Phi,\cF ) \stackrel{ \phi }{ \rightarrow } ( \Gamma,\cG )$ be
a $d$-pseudoflat epimorphism of $d$-homological pairs.  Let $n,n' \in \mod(
\Gamma )$ and $i \in \{ 0, \ldots, d \}$ be given.  There is an
isomorphism 
\[
\vcenter{
  \xymatrix {
    \Hom_{ \Db( \mod\,\Gamma ) }( n,\Sigma^i n' )
    \ar[rr]^-{ \phi_{ \astsmall }( - ) } & &
    \Hom_{ \Db( \mod\,\Phi ) }\!
      \big( 
        \phi_{ \ast }( n ),\Sigma^i \phi_{ \ast }( n' ) 
      \big).
            }
        }
\]
\end{Lemma}

\begin{proof}
The morphism in the lemma can be identified with the horizontal
morphism in the following commutative diagram.  
\[
\vcenter{
  \xymatrix {
    \Hom_{ \cD( \Gamma ) }( n,\Sigma^i n' )
      \ar[rrr]^-{ \RHom_{ \Gamma }( \Gamma,- ) }
      \ar[rrrdd]_-{ \Hom_{ \cD( \Gamma ) }( \varepsilon_{ n },\Sigma^i n' )\;\;\;\;\;\;\;\;\;\; } & & &
    \Hom_{ \cD( \Phi ) }\!
      \big( 
        \RHom_{ \Gamma }( \Gamma,n ),
        \RHom_{ \Gamma }( \Gamma,\Sigma^i n' )
      \big)
      \ar[dd]^-{\mbox{\rotatebox{90}{$\sim$}}} \\
    \\
    & & & \Hom_{ \cD( \Gamma ) }\!
      \big( 
        \RHom_{ \Gamma }( \Gamma,n )
        \LTensor{ \Phi } \Gamma,\Sigma^i n'
      \big)
            }
        }
\]
The vertical morphism is the adjunction isomorphism and $\varepsilon$
is the counit of the adjunction.  The diagram is commutative by
adjoint functor theory.  The lemma will follow if we can prove that
the diagonal morphism
$\Hom_{ \cD( \Gamma ) }( \varepsilon_{ n },\Sigma^i n' ) =
\varepsilon_{ n }^{ \astsmall }$
is an isomorphism.  This is true because the triangle in Lemma
\ref{lem:Tor_triangle} gives a long exact sequence containing
\[
  \Hom( x,\Sigma^{ i-1 }n' )
  \longrightarrow
  \Hom( n,\Sigma^i n' )
  \stackrel{ \varepsilon_{ n }^{ \astsmall } }{ \longrightarrow }
  \Hom\!
    \big( 
      \RHom_{ \Gamma }( \Gamma,n )
      \LTensor{ \Phi } \Gamma , \Sigma^i n'
    \big)
  \longrightarrow
  \Hom( x,\Sigma^i n' )
\]
where $\Hom$ means $\Hom_{ \cD( \Gamma ) }$, and where the outer terms
are zero since $i \in \{ 0, \ldots, d \}$ while $x$ satisfies
$\H_i( x ) = 0$ for $i \leqslant d$ by Lemma \ref{lem:Tor_triangle}.
\end{proof}

\section{Proof of Theorem A}
\label{sec:Proof_of_Thm_A}

This section proves Theorem A in the following steps: Proposition
\ref{pro:Psi1} says the map in Theorem A takes values in
functorially finite wide subcategories, Proposition \ref{pro:Psi2}
says it is surjective, and Proposition \ref{pro:Psi5} that it is
injective.

\begin{Setup}
In this section, $( \Phi,\cF )$ is a fixed $d$-homological pair.
\end{Setup}

\begin{Proposition}
\label{pro:Psi1}
If $( \Phi,\cF ) \stackrel{ \phi }{ \longrightarrow } ( \Gamma,\cG )$
is a $d$-pseudoflat epimorphism of $d$-homological pairs, then
$\phi_{ \bullet }( \cG )$ is a functorially finite wide
subcategory of $\cF$.
\end{Proposition}

\begin{proof}
To see that  $\phi_{ \bullet }( \cG )$ is wide,
we use Proposition \ref{pro:Proto_B} with $G = \phi_{ \ast }$.  We must check that conditions (a) and (b) in the proposition are satisfied. 

(a):  By Remark \ref{rmk:adjoints}, the functor $\phi_{ \ast }$ is
exact.  By Proposition \ref{pro:epimorphism}, it restricts to a full and faithful
functor $\cG \stackrel{ \phi_{ \bullet } }{ \longrightarrow } \cF$.

(b):  Holds by Lemmas \ref{lem:Ext_map} and
\ref{lem:Ext_isomorphism}.

To see that $\phi_{ \bullet }( \cG )$ is functorially finite, note that 
$\cG \stackrel{ \phi_{ \bullet } }{ \longrightarrow } \cF$ has left and right adjoint functors by Proposition \ref{pro:adjoints}.  Hence the inclusion of $\phi_{ \bullet }( \cG )$ into $\cF$ has left and right adjoint functors, whence $\phi_{ \bullet }( \cG )$ is functorially finite by Lemma \ref{lem:env} and its dual.
\end{proof}

\begin{Construction}
\label{con:Psi_inverse}
Let $\cW \subseteq \cF$ be a functorially finite wide subcategory.
We construct what will turn out to be a $d$-pseudoflat epimorphism of
$d$-homological pairs
$( \Phi,\cF ) \stackrel{ \phi }{ \rightarrow } ( \Gamma,\cG )$ with
$\phi_{ \bullet }( \cG ) = \cW$:

Since $\cW$ is functorially finite, it is enveloping in $\cF$.  Let
$ \cF \stackrel{ i^{ \astsmall } }{ \longrightarrow } \cW$ be a left
adjoint to the inclusion functor
$\cW \stackrel{ i_{ \astsmall } }{ \longrightarrow } \cF$, see
Proposition \ref{pro:envelopes_and_adjoints}(i).  Set
\begin{equation}
\label{equ:s_and_Gamma}
  s = i^{ \ast }( \Phi_{ \Phi } )
  \;\;,\;\;
  \Gamma = \End_{ \Phi }( s )
\end{equation}
so $s$ acquires the structure ${}_{ \Gamma }s_{ \Phi }$.  There is an algebra homomorphism
\[
  \End_{ \Phi }( \Phi_{ \Phi } )
  \stackrel{ i^{ \astsmall }( - ) }{ \longrightarrow }
  \End_{ \Phi }\!\big( i^{ \ast }( \Phi_{ \Phi } )  \big)
  =
  \End_{ \Phi }\!( s )
\]
which, up to canonical isomorphism of the source, is an algebra
homomorphism 
\[
  \Phi \stackrel{ \phi }{ \longrightarrow } \Gamma.
\]
By Remark \ref{rmk:adjoints} there is a functor
$\mod( \Phi ) \stackrel{ \phi^! }{ \longrightarrow } \mod( \Gamma )$
given by $\phi^!( - ) = \Hom_{ \Phi }( \Gamma,- )$, and we set
\[
\tag*{$\Box$}
  \cG = \phi^!( \cW ) \subseteq \mod( \Gamma ).
\]
\end{Construction}

\begin{Proposition}
\label{pro:Psi2}
Let $\cW \subseteq \cF$ be a functorially finite wide subcategory.
Construction \ref{con:Psi_inverse} gives a $d$-pseudoflat epimorphism of
$d$-homological pairs
$( \Phi,\cF ) \stackrel{ \phi }{ \rightarrow } ( \Gamma,\cG )$ with
$\phi_{ \bullet }( \cG ) = \cW$.
\end{Proposition}

\begin{proof}
The proof is divided into several claims.

Claim (i):  {\em The module $s$ and the algebra $\Gamma$ from
  Construction \ref{con:Psi_inverse} satisfy ${}_{ \Gamma }s_{ \Phi }
  \cong {}_{ \Gamma }\Gamma_{ \Phi }$. }

The unit $\eta$ of the adjunction $( i^{ \ast },i_{ \ast } )$ from
Construction \ref{con:Psi_inverse} gives a morphism
$\Phi_{ \Phi } \stackrel{ \eta_{ \Phi_{ \Phi } } }{ \longrightarrow }
i_{ \ast }i^{ \ast }( \Phi_{ \Phi } ) = s$.  Set
$u = \eta_{ \Phi_{ \Phi } }( 1 )$.  Define a map
\[
  \Gamma \stackrel{ \gamma }{ \longrightarrow } s
  \;\;,\;\;
  x \longmapsto x( u )
\]
which makes sense because $x \in \Gamma$ is an endomorphism of $s$.
Using that $\eta_{ \Phi_{ \Phi } }$ is a strong $\cW$-envelope by Lemma \ref{lem:env}, it is straightforward to check that $\gamma$
is a bijective homomorphism of $\Gamma$-$\Phi$-bimodules.

Claim (ii):  {\em $\cW$ is a projectively generated $d$-abelian
  category with $P( \cW ) = \add( s_{ \Phi } ) = \add( \Gamma_{ \Phi } )$. }

The equality $\add( s_{ \Phi } ) = \add( \Gamma_{ \Phi } )$ holds by
Claim (i).  The rest of Claim (ii) holds by Proposition
\ref{pro:W_is_n-abelian}(i) and Proposition \ref{pro:P_and_s}.

For the next claims, recall from Remark \ref{rmk:adjoints} that the
algebra homomorphism $\phi$ gives a restriction of scalars functor
$\phi_{ \ast }$ which has left and right adjoint functors:
\[
\vcenter{
\xymatrix
{
\mod( \Gamma )
    \ar[rrrr]^{ \phi_{ \astsmall }(-) } &&&&
  \mod( \Phi ).
    \ar@/_2.0pc/[llll]_{ \phi^{ \astsmall }(-) \,=\, - \underset{ \Phi }{ \otimes } \Gamma } 
    \ar@/^2.0pc/[llll]^{ \phi^!(-) \,=\, \Hom_{ \Phi }( \Gamma,- ) }
}
        }
\]

Claim (iii):  {\em There is a $d$-homological pair $( \Gamma,\cG
  )$ and an equivalence of categories 
$
  \xymatrix {
  \cW
    \ar[r]^-{ \phi^!|_{ \cW } } &
    \cG.
            }
$ 

}
By \cite[thm.\ 3.20]{J} and Claim (ii), the functor
\begin{equation}
\label{equ:Jasso_equivalence}
  \xymatrix {
  \cW
    \ar[rr]^-{ \cW( s,- ) } & &
    \mod( \Gamma )
            }
\end{equation}
is full and faithful, and its essential image is a $d$-cluster
tilting subcategory of $\mod( \Gamma )$.  Using Claim (i) we get
$\cW( s,- ) = \Hom_{ \Phi }( s,- )|_{ \cW } = \Hom_{ \Phi }( \Gamma,-
)|_{ \cW } = \phi^!( - )|_{ \cW }$, so \eqref{equ:Jasso_equivalence} 
can be identified with 
$
  \xymatrix {
  \cW
    \ar[r]^-{ \phi^!|_{ \cW } } &
    \mod( \Gamma ).
            }
$
Claim (iii) follows since $\phi^!( \cW ) = \cG$ by definition.

Claim (iv): {\em  Consider the adjunction
  $( \phi_{ \ast },\phi^!  )$.  The counit homomorphism
  $\phi_{ \ast }\phi^!( p ) \stackrel{ \varepsilon_p }{
    \longrightarrow } p$ is bijective for $p \in P( \cW )$. }

The homomorphism $\varepsilon_p$ is the canonical
homomorphism
\[
  \phi_{ \ast }\phi^!( p )
  = \Hom_{ \Phi }( \Gamma,p ) \underset{ \Gamma }{ \otimes } \Gamma
  \stackrel{ \varepsilon_p }{ \longrightarrow } p.
\]
Using Claim (i) it can be identified with the canonical homomorphism
\[
  \Hom_{ \Phi }( s,p ) \underset{ \Gamma }{ \otimes } s
  \stackrel{ \epsilon_p }{ \longrightarrow } p.
\]
However, $\epsilon_{ s_{ \Phi } }$ is an isomorphism because
$\Hom_{ \Phi }( s,s_{ \Phi } ) = \Gamma_{ \Gamma }$, cf.\ Equation
\eqref{equ:s_and_Gamma}.  Hence $\epsilon_p$ is an isomorphism for
$p \in \add( s_{ \Phi } )$, that is, for $p \in P( \cW )$ by Claim
(ii).

Claim (v):  {\em $\phi_{ \ast }$ and $\phi^!$ restrict to quasi-inverse
  equivalences } 
\[
  \xymatrix {
  \cG
    \ar[rr]<1ex>^-{ \phi_{ \astsmall }|_{ \cG } } & &
    \cW.
    \ar[ll]<1ex>^-{ \phi^!|_{ \cW } }
            }
\]

We already have the equivalence $\phi^!|_{ \cW }$ by Claim (iii).  Any
adjoint is a quasi-inverse, so $\phi_{ \ast }|_{ \cG }$ is a
quasi-inverse if it maps $\cG$ to $\cW$.  To see that it does, let
$g \in \cG$ be given.  Then
\[
  g \cong \phi^!( w )
\]
for some $w \in \cW$.  By Proposition \ref{pro:s_resolution} there is
an exact sequence $p_1 \rightarrow p_0 \rightarrow w \rightarrow 0$
with $p_i \in P( \cW )$ which stays exact when we apply the functor
$\Hom_{ \Phi }( p,- )$ with $p \in P( \cW )$.  By Claim (ii) it stays
exact when we apply the functor
$\phi^!( - ) = \Hom_{ \Phi }( \Gamma,- )$.  The functor $\phi_{ \ast
}$ is exact by Remark \ref{rmk:adjoints} so there is a commutative
diagram with exact rows,
\[
  \xymatrix {
    \phi_{ \ast }\phi^!( p_1 )
      \ar[r] \ar^{ \varepsilon_{ p_1 } }[d] &
    \phi_{ \ast }\phi^!( p_0 )
      \ar[r] \ar^{ \varepsilon_{ p_0 } }[d] &
    \phi_{ \ast }\phi^!( w )
      \ar[r] \ar^{ \varepsilon_{ w } }[d] &
    0 \\
    p_1 \ar[r] & p_0 \ar[r] & w \ar[r] & 0 \lefteqn{.}
            }
\]
The $\varepsilon_{ p_i }$ are bijective by Claim (iv), so
$\varepsilon_w$ is bijective which proves 
\[
  \phi_{ \ast }g \cong \phi_{ \ast }\phi^!( w ) \cong w \in \cW.
\]

Claim (vi):  {\em $( \Phi,\cF ) \stackrel{ \phi }{ \rightarrow } (
  \Gamma,\cG )$ is a $d$-pseudoflat epimorphism of $d$-homological pairs
with $\phi_{ \bullet }( \cG ) = \cW$. }

We already have the $d$-homological pair $( \Gamma,\cG )$ by Claim
(iii).  The algebra homomorphism $\phi$ is a morphism of
$d$-homological pairs since 
\begin{equation}
\label{equ:G_to_W}
  \phi_{ \ast }( \cG ) = \cW
\end{equation}
by Claim (v) whence $\phi_{ \ast }( \cG ) \subseteq \cF$.  The
functor $\phi_{ \bullet } = \phi_{ \ast }|_{ \cG }$ is full and
faithful by Claim (v), whence $\phi$ is an epimorphism of
$d$-homological pairs by Proposition \ref{pro:epimorphism}.  Equation
\eqref{equ:G_to_W} shows $\phi_{ \bullet }( \cG ) = \cW$.

Finally, there is an isomorphism
\[
  \Tor_d^{ \Phi }( \Gamma,\Gamma )
  \cong \dual\!\Ext^d_{ \Phi }( \Gamma,\dual\!\Gamma ) = ( \ast ),
\]
and we can write $( \ast )$ more elaborately as
\[
  ( \ast ) 
  = \dual\!\Ext_{ \Phi }^d\!
      \big( 
        \Gamma_{ \Phi },
        \phi_{ \ast }( ( \dual\!\Gamma )_{ \Gamma } )
      \big)
  = ( \ast\ast ).
\]
However, $\Gamma_{ \Phi } \in P( \cW )$ by Claim (ii), so $\Gamma_{ \Phi }
\in P_e( \cW )$ by Proposition \ref{pro:projectives}.  And
$( \dual\!\Gamma )_{ \Gamma }$ is injective so is in $\cG$ whence
$\phi_{ \ast }( ( \dual\!\Gamma )_{ \Gamma } ) \in \cW$ by Equation
\eqref{equ:G_to_W}.  But then $( \ast\ast ) = 0$ so $\phi$ is
$d$-pseudoflat.  
\end{proof}

\begin{Lemma}
\label{lem:Psi4}
Let
$( \Phi,\cF ) \stackrel{ \theta }{ \longrightarrow } ( \Lambda,\cL )$
be a $d$-pseudoflat epimorphism of $d$-homological pairs and set
$\cW = \theta_{ \bullet }( \cL )$.

Then $\cW$ is a functorially finite wide subcategory of $\cF$ by Proposition \ref{pro:Psi1}, and $\theta$ is equivalent to the $d$-pseudoflat epimorphism of
$d$-homological pairs $( \Phi,\cF ) \stackrel{ \phi }{ \longrightarrow } ( \Gamma,\cG )$ from Construction \ref{con:Psi_inverse}.
\end{Lemma}

\begin{proof}
The functor $\cL \stackrel{ \theta_{ \bullet } }{ \longrightarrow } \cF$ is full
and faithful by Proposition \ref{pro:epimorphism}.  We can factorise
$\theta_{ \bullet }$ as
$\cL \stackrel{ j }{ \longrightarrow } \cW \stackrel{ i_{ \astsmall } }{
  \longrightarrow } \cF$, where $i_{ \ast }$ is the inclusion functor
and $j$ is an equivalence of categories acting in the same way as
$\theta_{ \bullet }$, that is, taking a $\Lambda$-module and viewing
it as a $\Phi$-module through $\theta$.  There is a left adjoint of
$\theta_{ \bullet }$ given by
\[
  \theta^{ \bullet }( - )
  = - \underset{ \Phi }{ \otimes } \Lambda,
\]
and it induces a left adjoint of $i_{ \ast }$ given by
\[
  i^{ \ast }( - )
  = j\theta^{ \bullet }( - )
  = j( - \underset{ \Phi }{ \otimes } \Lambda ).
\]

There is a commutative diagram of algebra homomorphisms,
\[
  \xymatrix {
    & \End_{ \Lambda }( \Lambda_{ \Lambda } )
      \ar^{ \;j( - ) }_{ \mbox{\rotatebox{90}{$\backsim$}} }[dd] \\
    \End_{ \Phi }( \Phi_{ \Phi } )
      \ar^{ \theta^{ \bullet }( - ) }[ur]
      \ar_<<<<{ i^{ \ast }( - ) = j\theta^{ \bullet }( - ) \;\;\; }[dr] & \\
    & \End_{ \Phi }( \Lambda_{ \Phi } ) \lefteqn{.}
            }
\]
The vertical homomorphism is bijective because $j$ is an equivalence
of categories.  Up to canonical isomorphism, the ascending diagonal homomorphism is
$\Phi \stackrel{ \theta }{ \longrightarrow } \Lambda$.  The de\-scen\-ding
diagonal homomorphism is
$\Phi \stackrel{ \phi }{ \longrightarrow } \Gamma$ from Construction
\ref{con:Psi_inverse}.  Hence, up to canonical isomorphism, the
diagram is
\begin{equation}
\label{equ:Phi_Lambda_Gamma}
  \vcenter{
  \xymatrix {
    & & \Lambda
        \ar^{ \;\lambda }_{ \mbox{\rotatebox{90}{$\backsim$}} }[dd] \\
    \Phi
      \ar^{ \theta }[urr]
      \ar_{ \phi }[drr] & & \\
    & & \Gamma \lefteqn{,}
            }
          }
\end{equation}
where we write $\lambda = j( - )$.  

Proposition \ref{pro:Psi2} says that $\phi$ is a $d$-pseudoflat epimorphism
of $d$-homological pairs
$( \Phi,\cF ) \stackrel{ \phi }{ \longrightarrow } ( \Gamma,\cG )$
with $\cW = \phi_{ \bullet }( \cG )$, and we have
$\cW = \theta_{ \bullet }( \cL )$ by definition.  This means that
$\cW = \phi_{ \ast }( \cG ) = \theta_{ \ast }( \cL )$, and
since $\phi_{ \ast }$, $\theta_{ \ast }$, and $\lambda_{ \ast }$ are
full and faithful, \eqref{equ:Phi_Lambda_Gamma} then implies
$\lambda_{ \ast }( \cG ) = \cL$.
Hence $\lambda$ is an isomorphism
$( \Lambda,\cL ) \stackrel{ \lambda }{ \longrightarrow } ( \Gamma,\cG
)$ of $d$-homological pairs, which together with
\eqref{equ:Phi_Lambda_Gamma} gives the conclusion of the lemma.
\end{proof}

\begin{Proposition}
\label{pro:Psi5}
Let
$( \Phi,\cF ) \stackrel{ \phi_1 }{ \longrightarrow } ( \Gamma_1,\cG_1
)$ and
$( \Phi,\cF ) \stackrel{ \phi_2 }{ \longrightarrow } ( \Gamma_2,\cG_2
)$ be $d$-pseudoflat epimorphisms of $d$-homological pairs.  Then
$( \phi_1 )_{ \bullet }( \cG_1 ) = ( \phi_2 )_{ \bullet }( \cG_2 )$ if and
only if $\phi_1$ is equivalent to $\phi_2$.
\end{Proposition}

\begin{proof}
If $\phi_1$ is equivalent to $\phi_2$ then
$( \phi_1 )_{ \bullet }( \cG_1 ) = ( \phi_2 )_{ \bullet }( \cG_2 )$ is
immediate.

Conversely, if
$( \phi_1 )_{ \bullet }( \cG_1 ) = ( \phi_2 )_{ \bullet }( \cG_2 )$
then set $\cW = ( \phi_1 )_{ \bullet }( \cG_1 )$.  Lemma
\ref{lem:Psi4}(ii) shows that $\phi_1$ and $\phi_2$ are both equivalent to
the $d$-pseudoflat epimorphism of $d$-homological pairs
$( \Phi,\cF ) \stackrel{ \phi }{ \longrightarrow } ( \Gamma,\cG )$ from
Construction \ref{con:Psi_inverse}, whence $\phi_1$ is equivalent to
$\phi_2$.
\end{proof}

\section{Proof of Theorem C}
\label{sec:Proof_of_Thm_C}

This section uses Theorem B to prove Theorem C.  The following definition is used repeatedly.

\begin{Definition}
[Semisimple categories, see {\cite[sec.\ 1]{H}}]
\label{def:semi-simple}
We say that an additive category is {\em semisimple} if the endomorphism ring of each object is semisimple.
\hfill $\Box$
\end{Definition}

Recall that $d \geqslant 1$ is a positive integer.  Let $Q$ be the quiver
$m \to \cdots\to2 \to 1$, where $m \geqslant 3$. It is shown in
\cite[thm.\ 3]{V} that a quotient of $kQ$ by an admissible ideal $I$ has global
dimension $d$ and admits an $d$-cluster tilting subcategory $\cF$ if
and only if $I = (\rad kQ)^{\ell}$, where
\begin{equation}\label{fraction}
\frac{m-1} {\ell}  = \frac d 2
\end{equation}
and $d$ is even in case $\ell > 2$.  Moreover, in this case the objects in $\cF$ are
precisely the direct sums of projectives and injectives. Our goal is to
classify all wide subcategories of $\cF$. For ${\ell} = 2$ it is easy to
see that there are no non-semisimple wide subcategories of $\cF$
except $\cF$ itself. Among the semisimple subcategories, all are
wide except the ones containing both the simple projective and the
simple injective module. So this case is clear. In what follows we
will therefore assume that $\ell >2$ and $d$ is even so that the fraction in (\ref{fraction}) is an integer.

Let us begin by fixing some notation. Set
$\Phi = \Phi_{m,{\ell}} := kQ/(\rad kQ)^{\ell}$, where $Q$ and
${\ell}$ are as above. For each $i \in Q_0$ we denote the
corresponding indecomposable projective $\Phi$-module by $p_i$ and the
corresponding indecomposable injective $\Phi$-module by $q_i$. Note
that for ${\ell} \leqslant i \leqslant m$ we have
$p_i = q_{i-{\ell}+1}$.

We consider the basic $d$-cluster tilting module
\[
f  = \bigoplus_{i=1}^{m+{\ell}-1} f_i
\]
where $f_i = p_i$ for each $1 \leqslant i \leqslant m$ and
$f_i = q_{i-{\ell}+1}$ for each $m+1 \leqslant i \leqslant
m+{\ell}-1$. As mentioned we are interested in wide subcategories of
the $d$-cluster tilting subcategory
$\cF = \add (f) \subseteq \mod(\Phi)$. For an additive subcategory
$\cW \subseteq \cF$ we denote the number of isomorphism classes
of indecomposables in $\cW$ by $|\cW|$. In particular,
$|\cF| = m+{\ell}-1 = {\ell}(d+2)/2$. As a guiding example consider
the following.

\begin{Example}\label{Nak94} 
Let $d = 4$, ${\ell} =4$, and $m = 9$. Then $(9-1)/4 = 2 = 4/2$ so
that (\ref{fraction}) holds. Below we give the Auslander--Reiten
quiver of $\Phi_{9,4}$, with the modules $f_i$ indicated.

%%%AR quiver%%%
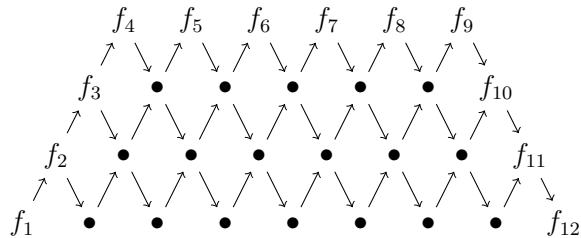
\[
\tikzstyle{nct} = [draw, circle, minimum size=.5cm, node distance=1.75cm]
\tikzstyle{nct2}=[draw, rectangle, minimum size=.5cm, node distance=1.75cm]
\begin{tikzpicture}[scale=0.9, every node/.style={transform shape}]
\node(11) at (1,1) {$f_1$};
\node(21) at (2,1) {$\bullet$};
\node(31) at (3,1) {$\bullet$};
\node(41) at (4,1) {$\bullet$};
\node(51) at (5,1) {$\bullet$};
\node(61) at (6,1) {$\bullet$};
\node(71) at (7,1) {$\bullet$};
\node(81) at (8,1) {$\bullet$};
\node(91) at (9,1) {$f_{12}$};

\node(12) at (1.5,2) {$f_2$};
\node(22) at (2.5,2) {$\bullet$};
\node(32) at (3.5,2) {$\bullet$};
\node(42) at (4.5,2) {$\bullet$};
\node(52) at (5.5,2) {$\bullet$};
\node(62) at (6.5,2) {$\bullet$};
\node(72) at (7.5,2) {$\bullet$};
\node(82) at (8.5,2) {$f_{11}$};

\node(13) at (2,3) {$f_3$};
\node(23) at (3,3) {$\bullet$};
\node(33) at (4,3) {$\bullet$};
\node(43) at (5,3) {$\bullet$};
\node(53) at (6,3) {$\bullet$};
\node(63) at (7,3) {$\bullet$};
\node(73) at (8,3) {$f_{10}$};

\node(14) at (2.5,4) {$f_4$};
\node(24) at (3.5,4) {$f_5$};
\node(34) at (4.5,4) {$f_6$};
\node(44) at (5.5,4) {$f_7$};
\node(54) at (6.5,4) {$f_8$};
\node(64) at (7.5,4) {$f_9$};

\draw[->] (11) -- (12);
\draw[->] (12) -- (13);
\draw[->] (13) -- (14);

\draw[->] (12) -- (21);
\draw[->] (13) -- (22);
\draw[->] (14) -- (23);

\draw[->] (22) -- (31);
\draw[->] (23) -- (32);
\draw[->] (24) -- (33);

\draw[->] (32) -- (41);
\draw[->] (33) -- (42);
\draw[->] (34) -- (43);

\draw[->] (42) -- (51);
\draw[->] (43) -- (52);
\draw[->] (44) -- (53);

\draw[->] (52) -- (61);
\draw[->] (53) -- (62);
\draw[->] (54) -- (63);

\draw[->] (62) -- (71);
\draw[->] (63) -- (72);

\draw[->] (72) -- (81);

\draw[->] (21) -- (22);
\draw[->] (22) -- (23);
\draw[->] (23) -- (24);

\draw[->] (31) -- (32);
\draw[->] (32) -- (33);
\draw[->] (33) -- (34);

\draw[->] (41) -- (42);
\draw[->] (42) -- (43);
\draw[->] (43) -- (44);

\draw[->] (51) -- (52);
\draw[->] (52) -- (53);
\draw[->] (53) -- (54);

\draw[->] (61) -- (62);
\draw[->] (62) -- (63);
\draw[->] (63) -- (64);

\draw[->] (71) -- (72);
\draw[->] (72) -- (73);

\draw[->] (81) -- (82);

\draw[->] (64) -- (73);
\draw[->] (73) -- (82);
\draw[->] (82) -- (91);

\end{tikzpicture}
\]
%%%
\end{Example}\hfill $\Box$

As a representation of $Q$ the module $f_i$ assigns to vertex $j$ the
vector space $k$ if $i-{\ell} +1 \leqslant j \leqslant i$ and $0$
otherwise. Moreover the corresponding maps $k \to k$ are all equal to
$1$. For example, as a representation, $f_6$ in Example \ref{Nak94} is 
\[
0 \to 0 \to 0 \to k \overset{1}{\to}  k \overset{1}{\to} k \overset{1}{\to} k \to 0 \to 0,
\]
whereas $f_3$ is
\[
0 \to 0 \to 0 \to 0 \to 0 \to 0 \to k \overset{1}{\to} k \overset{1}{\to} k.
\]

Considering these representations we readily deduce the following two basic observations on morphisms in $\cF$: First, the dimensions of morphism spaces in $\cF$ are given by the formula
\begin{equation}
\label{equ:morphisms}
\dim_k \Hom_\Phi(f_i,f_j) =
\begin{cases} 
1 & \mbox{if } 0 \leqslant j-i \leqslant {\ell}-1,\\ 
0 & \mbox{otherwise.}  
\end{cases}
\end{equation}
Secondly, if $i \leqslant q \leqslant j$, then each morphism from $f_i$ to $f_j$ factors through $f_q$. 

To compute extension groups we will employ higher Auslander--Reiten duality:
\[
\Ext^d_{\Phi}(f,f') \cong \dual\!\sHom_{\Phi}(\tau_d^-f',f),
\]
where $\tau_d^{-} := \tau^-\Omega^{-(d-1)}$ and $\tau^{-}$ is the inverse Auslander--Reiten translation.  By \cite[lem.\ 4.8]{V} we have
\[
\tau_d^{-}f_i = f_{i+m}
\]
for each $1 \leqslant i \leqslant {\ell}-1$.

Let us start by describing the semisimple wide subcategories of
$\cF$.  See Definition \ref{def:semi-simple} for the notion of
semisimplicity. 
\begin{Proposition}\label{semisimple}
Let $\cW \subseteq \cF$ be an additive subcategory. Then $\cW$ is semisimple and wide if and only if for all distinct $f_i,f_j \in \cW$ we have ${\ell} \leqslant |i-j| \leqslant m-1$.
\end{Proposition}
Combinatorially we can interpret Proposition \ref{semisimple} as
follows: If we identify the modules $f_i$ with the vertices in a
cyclic graph such that $f_i$ and $f_{i+1}$ are neighbours, then a
semisimple wide subcategory $\cW$ corresponds to a subset of vertices
where each vertex has distance at least ${\ell}$ to each of the others.

\begin{proof}
By our description of morphisms in $\cF$, the subcategory $\cW$ is
semisimple if and only if ${\ell} \leqslant |i-j|$ for all distinct
$f_i,f_j \in \cW$.  It follows by \cite[prop.\ A.1]{I3} that $\cW$ is
then wide if and 
only if $\Ext^d_{\Phi}(f_i,f_j) = 0$ for all $f_i,f_j \in \cW$. It is
enough to check this condition for $m < i$ and $j < {\ell}$
and by higher Auslander--Reiten duality it can be replaced by
$\dual\!\sHom_{\Phi}(f_{j+m},f_i) = 0$, which holds if and only if
$j+m > i$, i.e., $i-j \leqslant m-1$. Now $|i-j| > m-1$ is only
possible if $m < i$ and $j < {\ell}$, or $m < j$ and $i < {\ell}$, so we
conclude that $\cW$ is wide if and only if $|i-j| \leqslant m-1$ for
all $i$ and $j$.
\end{proof}

In order to deal with the non-semisimple subcategories we need to get some control over the $d$-exact sequences in $\cF$.  Let
$f_i \stackrel{ \mu }{ \rightarrow } f_j$ be a non-zero morphism where
$i \neq j$. Then $0 \leqslant j-i \leqslant {\ell}-1$. Note also that
$\mu$ is a monomorphism if and only if $j \leqslant {\ell}$ and an
epimorphism if and only if $i \geqslant m$. In all other cases we can
extend $f$ to an exact sequence
\[
f_{j-{\ell}} \overset{\lambda}\to f_i \overset{\mu}{\to} f_j \overset{\nu}\to f_{i+{\ell}}.
\]
So if we introduce the convention that $f_q = 0$ for each
$q \leqslant 0$ and each $q \geqslant m+{\ell}$, then we have the above exact
sequence even when $\mu$ is a monomorphism or an epimorphism. Note
that we do not care what the maps $\lambda$ and $\nu$ actually are, as
they are determined up to a non-zero scalar by Equation
\eqref{equ:morphisms}. 
Furthermore, the morphism $\lambda$ is in fact an
$\cF$-cover of $\Ker \mu$. Indeed, if
$\theta : f_q \to \Ker \mu$ is nonzero, then $q < m$ so $f_q$ is 
projective and $\theta$ factors through $\lambda$. Moreover, $\lambda$ is right minimal since $f_{j-{\ell}}$ is indecomposable
(or zero in case $\Ker \mu = 0$). Similarly, we find that $\nu$ is an
$\cF$-envelope of $\Cok \mu$.

Repeating this argument we obtain an acyclic complex 
\[
\mathbb{E}(\mu): \cdots \to f_{j-2{\ell}} \to f_{i-{\ell}} \to f_{j-{\ell}} \to f_i \overset{\mu}{\to} f_j \to f_{i+{\ell}} {\to} f_{j+{\ell}} \to f_{i+2{\ell}} \to \cdots
\]
Since $|\cF| = {\ell}(d+2)/2$, the complex $\mathbb{E}( \mu )$ has exactly $d+2$ non-zero terms and so it is a $d$-exact sequence of $\cF$. Truncating $\mathbb{E}( \mu )$ we find that
\[
\cdots \to f_{j-2{\ell}} \to f_{i-{\ell}} \to f_{j-{\ell}} \to f_i
\]
is the minimal $d$-kernel of $\mu$ and
\[
f_j \to f_{i+{\ell}} {\to} f_{j+{\ell}} \to f_{i+2{\ell}} \to \cdots
\]
is the minimal $d$-cokernel of $\mu$. With this understanding of $d$-kernels and $d$-cokernels we can prove the following result.

\begin{Lemma}\label{necessaryForWide}
Let $\cW \subseteq \cF$ be an additive subcategory which is not
semisimple.  Assume that $\cW$ is closed under minimal $d$-kernels and
minimal $d$-cokernels.  Then $0 \neq f_q \in \cW$ implies
$f_{q+r{\ell}} \in \cW$ for each $r \in \mathbb{Z}$. 
\end{Lemma}
\begin{proof}
Since $\cW$ is not semisimple there are $f_i, f_j \in \cW$ such that $1 \leqslant j-i \leqslant {\ell}-1$. By the discussion above we have that $f_{i+t{\ell}}, f_{j+t{\ell}} \in \cW$ for each $t \in \mathbb{Z}$. Pick $t = \lfloor \frac{q-i}{{\ell}} \rfloor$. Then we find that $f_{i+t{\ell}} \neq 0$ or $f_{i+t{\ell} +{\ell}} \neq 0$ and so there is a non-zero morphism $f_{i+t{\ell}} \to f_q$ or $f_q \to f_{i+t{\ell}+{\ell}}$. In either case $f_{q+r{\ell}} \in \cW$ for each $r \in \mathbb{Z}$.
\end{proof}

\begin{Definition}
[$\ell$-periodic subcategories]
\label{def:l-periodic}
We say that an additive subcategory $\cW \subseteq \cF$ is
${\ell}$-periodic if it is not semisimple and $0 \neq f_q \in \cW$
implies $f_{q+r{\ell}} \in \cW$ for each $r \in \mathbb{Z}$.
\hfill $\Box$
\end{Definition}
By Lemma \ref{necessaryForWide}, each non-semisimple wide
subcategory of $\cF$ is ${\ell}$-periodic. We will use Theorem B
to show that, conversely, each ${\ell}$-periodic subcategory is wide.
We begin with some additional notation.

For any $i,j$ let
$\cF_{ij}= \add\{f_q \mid i \leqslant q \leqslant j\}$.  Set
$\cW_{ij} = \cW \cap \cF_{ij}$ for each additive subcategory
$\cW \subseteq \cF$. Now assume that $\cW$ is ${\ell}$-periodic. Then
$\cW$ is uniquely determined by $\cW_{i,i+{\ell}-1}$ for each
$1 \leqslant i \leqslant m$. Moreover the number
${\ell}' := |\cW_{i,i+{\ell}-1}|$ is independent of $i$ and
$|\cW| = {\ell}'(d+2)/2$. Note that ${\ell}' \geqslant 2$ as $\cW$ is
not semisimple. Set $m' = {\ell}'d/2+1$ so that
$|\cW| = m'+{\ell}'-1$. Then there is a unique increasing map
$\iota : \{1,\ldots ,m'+{\ell}'-1\} \to \{1,\ldots ,m+{\ell}-1\}$ such
that $f' = \bigoplus_{i' = 1}^{m'+{\ell}'-1} f_{\iota(i')}$ is an
additive generator of $\cW$. Moreover, this map satisfies
\[
\iota(i'+{\ell}') = \iota(i')+{\ell}.
\]
Next we introduce a candidate for the object $s$ in Theorem B.
Set $s_{i'} = f_{\iota(i')}$ for each $1 \leqslant i' \leqslant m'$.
Further set
\[
  s = \bigoplus_{i' = 1}^{m'}s_{i'}
  \;\;,\;\;
  \Gamma = \End_{\Phi}(s).
\] 

\begin{Lemma}\label{endomorphismAlgebra} Assume that $\cW$ is ${\ell}$-periodic and use the notation above. Then
there is an algebra isomorphism $\Gamma \cong \Phi_{m',{\ell}'}$ and so there is a unique $d$-cluster tilting subcategory in $\mod( \Gamma )$. 
\end{Lemma}
\begin{proof}
First we observe that if $0 \leqslant j'-i' \leqslant {\ell}'-1$ then $\dim_k \Hom_\Phi(f_{\iota(i')}, f_{\iota(j')}) = 1$ because $f_{\iota(i')}, f_{\iota(j')} \in \cW_{\iota(i'),\iota(i')+{\ell}-1}$.  On the other hand, from $\iota(i'+{\ell}') = \iota(i')+{\ell}$ we obtain $\Hom_\Phi(f_{\iota(i')}, f_{\iota(i'+{\ell}')})=0$ and more generally $\Hom_\Phi(f_{\iota(i')}, f_{\iota(j')})=0$ if $j'-i' \geqslant {\ell}'$ or $i' > j'$. Furthermore if $i' \leqslant q' \leqslant j'$, then $\iota(i')\leqslant \iota(q') \leqslant \iota(j')$ so each morphism from $f_{\iota(i')}$ to $f_{\iota(j')}$ factors through $f_{\iota(q')}$. From this it follows that $\Gamma \cong \Phi_{m',{\ell}'}$.
\end{proof}
Now we are ready to show that $\cW$ is wide. To illustrate let us
consider Example \ref{Nak94}. Below we have indicated a $4$-periodic subcategory $\cW$ by encircling the corresponding indecomposables.
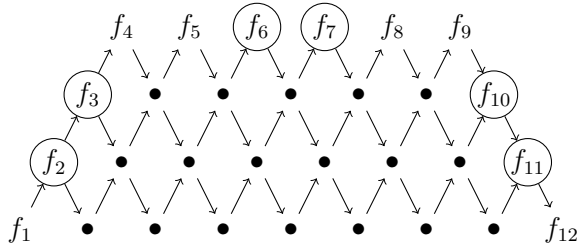
\[
\tikzstyle{nct} = [draw, circle, minimum size=.7cm, node distance=1.75cm]
\tikzstyle{nct2}=[draw, rectangle, minimum size=.5cm, node distance=1.75cm]
\begin{tikzpicture}[scale=0.9, every node/.style={transform shape}]
\node(11) at (1,1) {$f_1$};
\node(21) at (2,1) {$\bullet$};
\node(31) at (3,1) {$\bullet$};
\node(41) at (4,1) {$\bullet$};
\node(51) at (5,1) {$\bullet$};
\node(61) at (6,1) {$\bullet$};
\node(71) at (7,1) {$\bullet$};
\node(81) at (8,1) {$\bullet$};
\node(91) at (9,1) {$f_{12}$};

\node(12) at (1.5,2) {$f_2$};
\node[nct](12) at (1.5,2) {$ $};
\node(22) at (2.5,2) {$\bullet$};
\node(32) at (3.5,2) {$\bullet$};
\node(42) at (4.5,2) {$\bullet$};
\node(52) at (5.5,2) {$\bullet$};
\node(62) at (6.5,2) {$\bullet$};
\node(72) at (7.5,2) {$\bullet$};
\node(82) at (8.5,2) {$f_{11}$};
\node[nct](82) at (8.5,2) {$ $};

\node(13) at (2,3) {$f_3$};
\node[nct](13) at (2,3) {$ $};
\node(23) at (3,3) {$\bullet$};
\node(33) at (4,3) {$\bullet$};
\node(43) at (5,3) {$\bullet$};
\node(53) at (6,3) {$\bullet$};
\node(63) at (7,3) {$\bullet$};
\node(73) at (8,3) {$f_{10}$};
\node[nct](73) at (8,3) {$ $};

\node(14) at (2.5,4) {$f_4$};
\node(24) at (3.5,4) {$f_5$};
\node(34) at (4.5,4) {$f_6$};
\node[nct](34) at (4.5,4) {$ $};
\node(44) at (5.5,4) {$f_7$};
\node[nct](44) at (5.5,4) {$ $};
\node(54) at (6.5,4) {$f_8$};
\node(64) at (7.5,4) {$f_9$};

\draw[->] (11) -- (12);
\draw[->] (12) -- (13);
\draw[->] (13) -- (14);

\draw[->] (12) -- (21);
\draw[->] (13) -- (22);
\draw[->] (14) -- (23);

\draw[->] (22) -- (31);
\draw[->] (23) -- (32);
\draw[->] (24) -- (33);

\draw[->] (32) -- (41);
\draw[->] (33) -- (42);
\draw[->] (34) -- (43);

\draw[->] (42) -- (51);
\draw[->] (43) -- (52);
\draw[->] (44) -- (53);

\draw[->] (52) -- (61);
\draw[->] (53) -- (62);
\draw[->] (54) -- (63);

\draw[->] (62) -- (71);
\draw[->] (63) -- (72);

\draw[->] (72) -- (81);

\draw[->] (21) -- (22);
\draw[->] (22) -- (23);
\draw[->] (23) -- (24);

\draw[->] (31) -- (32);
\draw[->] (32) -- (33);
\draw[->] (33) -- (34);

\draw[->] (41) -- (42);
\draw[->] (42) -- (43);
\draw[->] (43) -- (44);

\draw[->] (51) -- (52);
\draw[->] (52) -- (53);
\draw[->] (53) -- (54);

\draw[->] (61) -- (62);
\draw[->] (62) -- (63);
\draw[->] (63) -- (64);

\draw[->] (71) -- (72);
\draw[->] (72) -- (73);

\draw[->] (81) -- (82);

\draw[->] (64) -- (73);
\draw[->] (73) -- (82);
\draw[->] (82) -- (91);

\end{tikzpicture}
\]
We have $\ell' = 2$, $m' = 5$.  The module $s$ is given by 
\[
s_1 = f_2 \;\;,\;\;
s_2 = f_3 \;\;,\;\;
s_3 = f_6 \;\;,\;\;
s_4 = f_7 \;\;,\;\;
s_5 = f_{10}.
\]
Consider the conditions of Theorem B.  Using higher Auslander--Reiten
duality one can rea\-di\-ly check that
$\Ext^{\geqslant 1}_{\Phi}(s,s) = 0$. The modules $f_2$, $f_3$, $f_6$,
$f_7$ and $f_{10}$ trivially permit exact sequences as in Theorem B,
condition (iii).  So it remains to deal with $f_{11}$, which has the
exact sequence
\begin{equation}
\label{equ:Martins_sequence}
  0 \to s_1\to s_2 \to s_3 \to s_4 \to s_5 \to f_{11} \to 0.
\end{equation}
Now consider the Auslander--Reiten quiver of $\Gamma  \cong \Phi_{5,2}$ with the indecomposables of its $d$-cluster tilting subcategory $\cG$ labelled $f'_{i'}$ for $1 \leqslant i' \leqslant 6$.

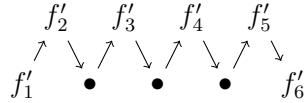
\[
\tikzstyle{nct} = [draw, circle, minimum size=.5cm, node distance=1.75cm]
\tikzstyle{nct2}=[draw, rectangle, minimum size=.5cm, node distance=1.75cm]
\begin{tikzpicture}[scale=0.9, every node/.style={transform shape}]
\node(11) at (1,1) {$f'_{1}$};
\node(21) at (2,1) {$\bullet$};
\node(31) at (3,1) {$\bullet$};
\node(41) at (4,1) {$\bullet$};
\node(51) at (5,1) {$f'_{6}$};

\node(12) at (1.5,2) {$f'_{2}$};
\node(22) at (2.5,2) {$f'_{3}$};
\node(32) at (3.5,2) {$f'_{4}$};
\node(42) at (4.5,2) {$f'_{5}$};

\draw[->] (11) -- (12);
\draw[->] (12) -- (21);
\draw[->] (22) -- (31);
\draw[->] (32) -- (41);
\draw[->] (42) -- (51);
\draw[->] (21) -- (22);
\draw[->] (31) -- (32);
\draw[->] (41) -- (42);

\end{tikzpicture}
\]
By definition $\Hom_{\Phi}(s,s_{i'}) = f'_{i'}$. Moreover, applying $\Hom_{\Phi}(s,-)$ to the sequence \eqref{equ:Martins_sequence},
we obtain
\[
0 \to f'_1\to f'_2 \to f'_3 \to f'_4 \to f'_5 \to \Hom_{\Phi}(s,f_{11}) \to 0 
\]
which is exact since $\Ext^{\geqslant 1}_{\Phi}(s,s) = 0$. This is in fact the minimal projective resolution of $f'_6$ so in particular, $\Hom_{\Phi}(s,f_{11}) = f'_6$, which shows that $\Hom_\Phi(s,\cW) = \cG$. Next we turn to the general case.

\begin{Lemma}
\label{lem:periodic_implies_wide}
Any ${\ell}$-periodic subcategory $\cW \subseteq \cF$ is wide.
\end{Lemma}
\begin{proof}
We define $s$ as above.
As in the example it is enough to show that Theorem B applies. Let us
check conditions (i)--(iv) in the theorem. 

(i):  $\Phi$ has finite global dimension so $s$ has finite projective
dimension. 

(ii):  Since $\cF$ is $d$-cluster tilting we only need to show that
$\Ext^d_{\Phi}(s,s_{i'}) = 0$ for each $i'$. To do this we apply higher
Auslander--Reiten duality
$\Ext^d_{\Phi}(s,s_{i'}) \cong \dual\!\sHom_{\Phi}(\tau_d^-s_{i'},
s)$. Using $\tau_d^-f_{i} = f_{i+m}$ and
$(m-1)/{\ell} = d/2=(m'-1)/{\ell}'$ we compute that
$\tau_d^-s_{i'} = \tau_d^-f_{\iota(i')} = f_{\iota(i')+m-1+1} =
f_{\iota(i'+m'-1)+1}$. Since $\iota(i'+m'-1)+1 > \iota(m')$ we get
$\Hom_\Phi(f_{\iota(i'+m'-1)+1},s) =0$ and so
$\Ext^d_{\Phi}(s,s_{i'}) \cong \dual\!\sHom_{\Phi}(\tau_d^-s_{i'}, s) = 0$.

(iii):  We only need to construct the exact sequences required in
condition (iii) for each indecomposable $f_{\iota(i')} \in \cW$. For
$1 \leqslant i' \leqslant m'$ we have $s_{i'} = f_{\iota(i')}$ so
there is nothing to do. Let $m'+1\leqslant i' \leqslant
m'+{\ell}'-1$. Since $s_{m'} = f_{\iota(m')} \in \cW_{m,m+{\ell}-1}$,
there is an epimorphism $\mu : s_{m'} \to f_{\iota(i')}$. By our
previous discussion on $d$-extensions, we can extend $\mu$ to get the
desired exact sequence,
\begin{equation}
\label{equ:desired_exact_sequence}
0 \to s_{1} \to s_{i'-m'+1} \to \cdots \to s_{i'-{\ell}'} \to s_{m'}
\overset{\mu}\to f_{\iota(i')} \to 0.
\end{equation}

(iv):  By definition,
\[
  \cG = \add \{ \Hom_{ \Phi }( s,f_{ \iota( i' ) } )
                \mid 1 \leqslant i' \leqslant m'+\ell'-1 \}.
\]
On the other hand, since $\Gamma \cong \Phi_{m'{\ell}'}$ we know that
the unique $d$-cluster tilting subcategory of $\mod( \Phi_{ m'{\ell}'
} )$ is
\[
  \cG' = \add\{ f'_{i'} \mid 1 \leqslant i' \leqslant m'+{\ell}'-1\}
\]
where $f'_{i'} = \Hom_\Phi(s,s_{i'})$ for each
$1 \leqslant i' \leqslant m'$ and
$f'_{i'} = \dual\!\Hom_\Phi(s_{i'-{\ell}'+1},s)$ for each
$m'+1 \leqslant i' \leqslant m'+{\ell}'-1$. We will show $\cG = \cG'$
by showing that $\Hom_\Phi(s,f_{\iota(i')}) \cong f'_{i'}$. This is
true by definition for $1 \leqslant i' \leqslant m'$, so we may assume
$m'+1\leqslant i' \leqslant m'+{\ell}'-1$.

By \cite[prop.\ 2.2]{JK} and the fact $\Ext^d_{\Phi}(s,s_1) = 0$, we
find that applying $\Hom_{\Phi}(s,-)$ to the exact sequence
\eqref{equ:desired_exact_sequence} from part (iii) yields an exact
sequence
\[
0 \to \Hom_{\Phi}(s,s_{1}) \to \Hom_{\Phi}(s,s_{i'-m'+1}) \to \cdots
\to \Hom_{\Phi}(s,s_{m'}) \to \Hom_{\Phi}(s,f_{\iota(i')}) \to 0.
\]
Dropping the last term in this sequence gives
\[
0 \to f'_{1} \to f'_{i'-m'+1} \to \cdots \to f'_{i'-{\ell}'} \to f'_{m'}.
\]
Now our discussion on $d$-exact sequences applies also to $\Gamma \cong \Phi_{m'{\ell}'}$ so we have an exact sequence
\[
0 \to f'_{1} \to f'_{i'-m'+1} \to \cdots \to f'_{i'-{\ell}'} \to f'_{m'} \to f'_{i'} \to 0
\]
whence $\Hom_{\Phi}(s,f_{\iota(i')}) \cong f'_{i'}$. 
\end{proof}

We can now finally prove Theorem C.

{\em Proof }(of Theorem C).
Lemma \ref{necessaryForWide} shows that if $\cW$ is wide, then
$\cW$ is ${\ell}$-periodic. Conversely, if $\cW$ is ${\ell}$-periodic,
then it is wide by Lemma \ref{lem:periodic_implies_wide}.  Finally
each ${\ell}$-periodic subcategory $\cW$ is determined by an arbitrary
choice of ${\ell}' \geqslant 2$ indecomposables from
$\cF_{1{\ell}}$. Evidently there are $2^{\ell}-{\ell}-1$ such choices.
\hfill $\Box$

\medskip
\noindent
{\bf Acknowledgement.}
We thank the referee for their careful reading and useful comments.  This work was supported by EPSRC grant EP/P016014/1 ``Higher Dimensional Homological Algebra''.

\end{document}